\documentclass[a4paper]{article}
\usepackage{latexsym,amsmath} 
\usepackage{times}

\def\vec#1{\mathchoice{\mbox{\boldmath$\displaystyle#1$}}
{\mbox{\boldmath$\textstyle#1$}}
{\mbox{\boldmath$\scriptstyle#1$}}
{\mbox{\boldmath$\scriptscriptstyle#1$}}}

\newcommand{\qed}{\hfill$\Box$\smallskip}
\newenvironment{proof}{\emph{Proof.}}{}

\newtheorem{definition}{Definition}[section]

\newtheorem{theorem}[definition]{Theorem}
\newtheorem{lemma}[definition]{Lemma}
\newtheorem{proposition}[definition]{Proposition}
\newtheorem{corollary}[definition]{Corollary}

\newtheorem{algorithm}[definition]{Algorithm}
\newtheorem{fact}[definition]{Fact}
\newtheorem{hypothesis}[definition]{Hypothesis}
\newtheorem{experiment}[definition]{Experiment}
\newtheorem{Claim}[definition]{Claim}

\usepackage[dvips]{epsfig} 
\graphicspath{{./Fig_eps/}{./Eps/}} 
\DeclareGraphicsExtensions{.ps,.eps} 
\usepackage{color} 
\usepackage{fullpage}

\newcommand\Paragraph[1]{\noindent{\bf\em #1}}

\newcommand{\dist}{\mbox{dist}}

\newcommand\sign{\mathrm{sign}}
\newcommand\id{\mathrm{id}}
\newcommand\BP{\mathrm{BP}}

\newcommand\PHI{\vec\Phi} 
\newcommand\SIGMA{\vec\sigma} 

\newcommand\cutnorm[1]{\left\|{#1}\right\|_{\Box}} 
\newcommand\cA{\mathcal{A}} 
 
\newcommand\cC{\mathcal{C}} 
 
\newcommand\cF{\mathcal{F}} 
\newcommand\cG{\mathcal{G}} 
\newcommand\cE{\mathcal{E}} 
\newcommand\cU{\mathcal{U}}

\newcommand\cS{\mathcal{S}} 
 
\newcommand\cI{\mathcal{I}}

\newcommand\cL{\mathcal{L}} 
 
\newcommand\cP{\mathcal{P}}

\def\cC{{\mathcal C}}
\def\cE{{\cal E}}

\newcommand\eul{\mathrm{e}} 
\newcommand\eps{\varepsilon}

\newcommand\Erw{\mathrm{E}} 
\newcommand\pr{\mathrm{P}}

\newcommand{\vecone}{\vec{1}}

\newcommand\EE{\mathcal{E}}

\newcommand{\Bin}{{\rm Bin}}

\newcommand{\bink}[2] {{{#1}\choose {#2}}}

\newcommand\ra{\rightarrow} 

\newcommand\bc[1]{\left({#1}\right)} 
\newcommand\cbc[1]{\left\{{#1}\right\}} 
\newcommand\bcfr[2]{\bc{\frac{#1}{#2}}} 
 
\newcommand\brk[1]{\left\lbrack{#1}\right\rbrack} 
 
\newcommand\norm[1]{\left\|{#1}\right\|} 
\newcommand\abs[1]{\left|{#1}\right|}

\newcommand\RR{\mathbf{R}} 
 
\newcommand{\Whp}{W.h.p.} 
\newcommand{\whp}{w.h.p.}

\newcommand{\Mezard}{M\'ezard}
\newcommand\form{F}

\newcommand\Lem{Lemma}
\newcommand\Prop{Proposition}
\newcommand\Thm{Theorem}
\newcommand\Cor{Corollary}
\newcommand\Sec{Section}

\newcommand\algstyle{\small\sffamily}

\begin{document} 

\title{The decimation process in random $k$-SAT}

\author{
Amin Coja-Oghlan and Angelica Y.~Pachon-Pinzon\thanks{Supported by EPSRC grant EP/G039070/2 and DIMAP.}\\
University of Warwick, Mathematics and Computer Science,\\
Zeeman building, Coventry CV4~7AL, UK\\
	\texttt{$\{$a.coja-oghlan,a.y.pachon-pinzon$\}$@warwick.ac.uk}} 
\date{\today}

\maketitle 

\begin{abstract}
Let $\PHI$ be a uniformly distributed random $k$-SAT formula with $n$ variables and $m$ clauses.
Non-rigorous statistical mechanics ideas have inspired a message passing algorithm
called \emph{Belief propagation guided decimation} for finding satisfying assignments of $\PHI$.
This algorithm can be viewed as an attempt at implementing a certain thought experiment
that we call the \emph{decimation process}.
In this paper we identify a variety of phase transitions in the decimation process and
link these phase transitions to the performance of the
algorithm.

\emph{Key words:}	random structures, phase transitions, $k$-SAT, Belief Propagation.\\
\end{abstract}

\section{Introduction}
\enlargethispage{1cm}

Let $k\geq3$ and $n>1$ be integers, let $r>0$ be a real, and set $m=\lceil rn\rceil$.
Let $\PHI=\PHI_k(n,m)$ be a propositional formula obtained
by choosing a set of $m$ clauses of length $k$ over the variables $V=\cbc{x_1,\ldots,x_n}$ uniformly at random.
For $k,r$ fixed we say that $\PHI$ has some property $\cP$ \emph{with high probability} (`\whp') if
$\lim_{n\ra\infty}\pr\brk{\PHI\in\cP}=1$.

The interest in random $k$-SAT originates from the experimental observation
that for certain densities $r$ the random formula $\PHI$ is satisfiable \whp\
while a large class of algorithms, including and particularly the workhorses of practical SAT solving such
as sophisticated DPLL-based solvers, fail to find a satisfying assignment efficiently~\cite{MitchellSelmanLevesque}.
Over the past decade, a fundamentally new class of algorithms has been proposed
on the basis of ideas from statistical physics~\cite{BMZ,MPZ}.
Experiments performed for $k=3,4,5$ indicate that these new `message passing algorithms', namely \emph{Belief Propagation guided decimation}
and \emph{Survey Propagation guided decimation} (`BP/SP decimation'), excel on random $k$-SAT instances~\cite{Kroc}.
Indeed, the experiments indicate that
BP/SP decimation find satisfying assignments for $r$ close to the threshold where $\PHI$ becomes unsatisfiable \whp\
Generally, SP decimation is deemed conceptually superior to BP decimation.

For example, in the case $k=4$ 
the threshold for the existence of satisfying assignments is conjectured to be $m/n\sim r_4\approx9.93$~\cite{Mertens}.
According to experiments from~\cite{Kroc}, SP decimation finds satisfying assignments for densities up to $r=9.73$.
Experiments from~\cite{RTS} suggest that the ``vanilla'' version of BP decimation
succeeds up to $r=9.05$.
Another version of BP decimation (with a different decimation strategy from~\cite{BMZ})
succeeds up to $r=9.24$, again according to experimental data from~\cite{Kroc}.
By comparison, the currently best rigorously
analyzed algorithm is efficient up to $r=5.54$~\cite{FrSu},
while {\tt zChaff}, a prominent practical SAT solver, becomes ineffective beyond $r=5.35$~\cite{Kroc}.

Since random $k$-SAT instances have widely been deemed extremely challenging benchmarks,
the stellar experimental performance of the physicists' message passing algorithms has stirred considerable excitement.
However, the statistical mechanics ideas that BP/SP decimation are based
on are highly non-rigorous, and thus a rigorous analysis
of these message passing algorithms is an important but challenging open problem.
A first step was made in~\cite{BPdec},
where it was shown that BP decimation does not outperform far simpler
combinatorial algorithms for sufficiently large clause lengths $k$.
More precisely, the main result of~\cite{BPdec} is that there is a constant $\rho_0>0$ (independent of $k$) such that
the `vanilla' version of BP decimation fails to find satisfying assignments
\whp\ if $r>\rho_02^k/k$.
By comparison, 
non-constructive arguments show that \whp\  $\PHI$ is satisfiable if
	$r<r_k=2^k\ln 2-k$, and unsatisfiable if $r>2^k\ln2$~\cite{nae,yuval}.
This means that for $k\gg\rho_0$ sufficiently large, BP decimation fails
to find satisfying assignments \whp\ already for densities a factor of (almost) $k$ below
the threshold for satisfiability.

The analysis performed in~\cite{BPdec} is based on an intricate method for directly tracking the execution of BP decimation.
Unfortunately this argument does little to illuminate the conceptual reasons for the algorithms' demise.
In particular, \cite{BPdec} does not provide a link to the statistical mechanics ideas that inspired the algorithm.
The present paper aims to remedy these defects.
Here we study the \emph{decimation process}, an idealized thought experiment that the BP decimation algorithm aims to implement.
We show that this experiment undergoes a variety of phase transitions that explain the failure of BP decimation
for densities $r>\rho_0\cdot 2^k/k$.
Our results identify phase transitions jointly in terms of the clause/variable density $r$ and with respect to the time parameter
of the decimation process.
The latter dimension was ignored in the original statistical mechanics work on BP~\cite{BMZ,MPZ} but turns
out to have a crucial impact on the performance of the algorithm.
On a non-rigorous basis, this has been pointed out recently by Ricci-Tersenghi and Semerjian~\cite{RTS}, and our results can
be viewed as providing a rigorous version of (substantial parts of) their main results.
The results of this paper can also be seen as a generalization of the ones obtained in~\cite{Barriers} for random $k$-SAT,
and indeed our proofs build upon the techniques developed in that paper.

\section{Results}\label{Sec_results}

BP decimation is a polynomial-time algorithm that aims to (heuristically) implement the `thought experiment'
shown in Fig.~\ref{Fig_dec}~\cite{Allerton,RTS},
which we call the \emph{decimation process}.\footnote{Several
different versions of BP decimation have been suggested.
In this paper we refer to the simplest but arguably most natural one, also considered in~\cite{BPdec,Allerton,RTS}.
Other versions decimate the variables in a different order, allowing for slightly better experimental
results~\cite{BMZ,Kroc}.}
A moment's reflection reveals that, given a satisfiable input formula $\Phi$,
the decimation process outputs a uniform sample from the set of all satisfying assignments of $\Phi$.
The obvious obstacle to actually implementing this experiment is the computation of the marginal probability $M_{x_t}(\Phi_{t-1})$
that $x_t$ takes the value `true' in a random satisfying assignment of $\Phi_{t-1}$,
a $\#P$-hard problem in the worst case.
Yet the key hypothesis underlying BP decimation is that
these marginals can be computed efficiently on \emph{random} formulas
by means of a message passing algorithm.
We will return to the discussion of BP decimation and its
connection to Experiment~\ref{Exp_dec} below.

\begin{figure}
\noindent{
\begin{experiment}[`decimation process']\label{Exp_dec}\upshape
\emph{Input:} A satisfiable $k$-CNF $\Phi$.\\
\emph{Result:} A satisfying assignment $\sigma:V\rightarrow\cbc{0,1}$
	(with $0/1$ representing `false'/`true').
\begin{tabbing}
mmm\=mm\=mm\=mm\=mm\=mm\=mm\=mm\=mm\=\kill
{\algstyle 0.}	\> \parbox[t]{36em}{\algstyle
	Let $\Phi_0=\Phi$.}\\
{\algstyle 1.}	\> \parbox[t]{36em}{\algstyle
	For $t=1,\ldots,n$ do}\\
{\algstyle 2.}	\> \> \parbox[t]{34em}{\algstyle
		Compute the fraction $M_{x_t}(\Phi_{t-1})$ of all satisfying assignments of $\Phi_{t-1}$
				in which the variable $x_t$ takes the value $1$.}\\
{\algstyle 3.}	\> \> \parbox[t]{34em}{\algstyle
		Assign $\sigma(x_t)=1$ with probability $M_{x_t}(\Phi_{t-1})$, and let $\sigma(x_t)=0$ otherwise.}\\
{\algstyle 4.}	\> \> \parbox[t]{34em}{\algstyle
		Obtain the formula $\Phi_t$ from $\Phi_{t-1}$ by substituting the value $\sigma(x_t)$ for $x_t$ and simplifying
		(i.e., delete all clauses that got satisfied by assigning $x_t$, and omit $x_t$ from all other clauses).}\\
{\algstyle 5.}	\> \parbox[t]{36em}{\algstyle
	Return the assignment $\sigma$.}
\end{tabbing}
\end{experiment}
}
\caption{The decimation process.}\label{Fig_dec}
\end{figure}

We are going to study the decimation process when applied to a random formula $\PHI$
for densities $r<2^k\ln2-k$, i.e., in the regime where $\PHI$ is satisfiable \whp\
More precisely, conditioning on $\PHI$ being satisfiable, we let $\PHI_t$ be the (random) formula obtained
after running the first $t$ iterations of Experiment~\ref{Exp_dec}.
The variable set of this formula is $V_t=\cbc{x_{t+1},\ldots,x_n}$,
and each clause of $\PHI_t$ consists of \emph{at most} $k$ literals.
Let $\cS(\PHI_t)\subset\cbc{0,1}^{V_t}$ be the set of all satisfying assignments of $\PHI_t$.
We say that \emph{almost all} $\sigma\in\cS(\PHI_t)$ have a certain property $\cA$
if $|\cA\cap\cS(\PHI_t)|=(1-o(1))|\cS(\PHI_t)|$.

We will identify various phase transition that the
formulas $\PHI_t$ undergo as $t$ grows from $1$ to $n$.
As it turns out, these can be characterized via two simple parameters.
The first one is the clauses density $r\sim m/n$.
Actually, it will be most convenient to work in terms of
	$$\rho=kr/2^k,$$
so that $m/n\sim\rho\cdot 2^k/k$.
We will be interested in the regime $\rho_0\leq \rho\leq k\ln2$, where $\rho_0$ is a constant (independent of $k$).
The upper bound $k\ln 2$ marks the point where satisfying assignments cease to exist~\cite{yuval}.
The second parameter is the fraction
	$$\theta=1-t/n$$
of `free' variables (i.e., variables not yet assigned by time $t$).

\Paragraph{The symmetric phase.}
Let $\Phi$ be a $k$-CNF on $V$, let $1\leq t< n$, let $\Phi_t$ be the formula obtained after $t$ steps of
the decimation process, and suppose that $\sigma\in\cS(\Phi_t)$.
A variable $x\in V_t$ is \emph{loose} if there is $\tau\in\cS(\Phi_t)$ such that
$\sigma(x)\neq\tau(x)$ and $d(\sigma,\tau)\leq\ln n$, where $d(\cdot,\cdot)$ denotes the Hamming distance.
For any $x\in V_t$ we let
	$$M_x(\Phi_t)=\frac{\abs{\cbc{\sigma\in \cS(\Phi_t):\sigma(x)=1}}}{|\cS(\Phi_t)|}$$
be the marginal probability that $x$ takes the value `true' in a random satisfying assignment of $\Phi_t$.

\begin{theorem}\label{XThm_sym}
There are constants $k_0,\rho_0>0$ such that for $k\geq k_0$, $\rho_0\leq\rho\leq k\ln2-2\ln k$, and
	$$k\cdot\theta>\exp\brk{\rho\bc{1+\frac{\ln\ln\rho}{\rho}+\frac{10}{\rho}}}$$
the random formula $\PHI_t$ has the following properties \whp
\begin{enumerate}
\item In almost all satisfying assignments $\sigma\in\cS(\PHI_t)$ at least $0.99\theta n$ variables are loose.
\item At least $\theta n/3$ variables $x\in V_t$ satisfy $M_x(\PHI_t)\in\brk{0.01,0.99}$.
\item The average distance of two random satisfying assignments satisfies
		$$\sum_{\sigma,\tau\in\cS(\PHI_t)}d(\sigma,\tau)/|\cS(\PHI_t)|^2\geq0.49\theta n.$$
\end{enumerate}
\end{theorem}

Intuitively, \Thm~\ref{XThm_sym} can be summarized as follows.
In the early stages of the decimation process (while $\theta$ is `big'),
most variables in a typical $\sigma\in\cS(\PHI_t)$ are loose.
Hence, the correlations amongst the variables
are mostly local: if we `flip' one variable in $\sigma$, then we can `repair' the unsatisfied clauses that this may cause
by simply flipping another $\ln n$ variables.
Furthermore, for at least a good fraction of the variables, the marginals $M_x(\PHI_t)$ are bounded away from $0/1$.
Finally, as the average distance between satisfying assignments is large on average, the set $\cS(\PHI_t)$ is `well spread'
over the Hamming cube $\cbc{0,1}^{V_t}$.

\Paragraph{Shattering and rigidity.}
Let $\Phi$ be a $k$-CNF and let $\sigma\in\cS(\Phi_t)$.
For an integer $\omega\ge1$ we call a variable $x\in V_t$ \emph{$\omega$-rigid} if any $\tau\in\cS(\Phi_t)$
with $\sigma(x)\neq\tau(x)$ satisfies $d(\sigma,\tau)\geq\omega$.

Furthermore, we say that a set $S\subset\cbc{0,1}^{V_t}$ is \emph{$(\alpha,\beta)$-shattered} if it admits a decomposition
	$S=\bigcup_{i=1}^NR_i$ into pairwise disjoint subsets such that the following two conditions are satisfied.
\begin{description}
\item[SH1.] We have $|R_i|\leq\exp(-\alpha\theta n)|S|$ for all $1\leq i\leq N$.
\item[SH2.] If $1\leq i<j\leq N$ and $\sigma\in R_i$, $\tau\in R_j$, then $\dist(\sigma,\tau)\geq\beta\theta n$.
\end{description}

\begin{theorem}\label{XThm_dRSB}
There are constants $k_0,\rho_0>0$ such that for $k\geq k_0$, $\rho_0\leq\rho\leq k\ln2-2\ln k$, and
	\begin{equation}\label{XeqdRSB}
	\frac{\rho}{\ln2}(1+2\rho^{-2})\leq k\theta\leq\exp\brk{\rho\bc{1-\frac{\ln\rho}\rho-\frac2\rho}}
	\end{equation}
the random formula $\PHI_t$ has the following properties \whp
\begin{enumerate}
\item In almost all $\sigma\in\cS(\PHI_t)$ at least $0.99\theta n$ variables are $\Omega(n)$-rigid.
\item There exist $\alpha=\alpha(k,\rho)>0,\beta=\beta(k,\rho)>0$ such that
		 $\cS(\PHI_t)$ is $(\alpha,\beta)$-shattered.
\item At least $\theta n/3$ variables $x\in V_t$ satisfy $M_x(\PHI_t)\in\brk{0.01,0.99}$.
\item The average distance of two random satisfying assignments is at least
		$0.49\theta n.$
\end{enumerate}
\end{theorem}

Thus, if the fraction $\theta$ of free variables lies in the regime~(\ref{XeqdRSB}),
then in most satisfying $\sigma\in\cS(\PHI_t)$ the values assigned to $99\%$ of the variables
are linked via long-range correlations: to `repair' the damage done by flipping a single rigid variable
it is inevitable to reassign a \emph{constant fraction} of all variables.
This is mirrored in the geometry of the set $\cS(\PHI_t)$: it decomposes into exponentially
many exponentially tiny subsets, which are mutually separated by a linear Hamming distance $\Omega(n)$.
Yet as in the symmetric phase, the marginals of a good fraction of the free variables remain bounded away from $0/1$,
and the set $\cS(\PHI_t)$ remains `well spread' over the Hamming cube $\cbc{0,1}^{V_t}$.

\Paragraph{The condensation phase.}
Let $\alpha>0$.
We say that a set $S\subset\cbc{0,1}^{\theta n}$ is {\em $\alpha$-condensed} if 
for any $\sigma,\tau\in S$ we have 
	$\dist(\sigma,\tau)\leq\alpha n$.

\begin{theorem}\label{XThm_cond}
There are constants $k_0,\rho_0>0$ such that for $k\geq k_0$, $\rho_0\leq\rho\leq k\ln2-2\ln k$, and
	\begin{equation}\label{Xeqcond}
	\ln\rho <k\cdot\theta<(1-\rho^{-2})\cdot\rho/\bc{\ln 2}
	\end{equation}
the random formula $\PHI_t$ has the following properties \whp
\begin{enumerate}
\item In almost all $\sigma\in\cS(\PHI_t)$ at least $0.99\theta n$ variables are $\Omega(n)$-rigid.
\item The set $\cS(\PHI_t)$ is $\exp(2-\rho)/k$-condensed.
\item At least $0.99\theta n$ variables $x\in V_t$ satisfy $M_x(\PHI_t)\in\brk{0,2^{-k/2}}\cup\brk{1-2^{-k/2},1}$.
\item There is a set $R\subset V_t$ of size $|R|\geq0.99\theta n$ such that
		for any $\sigma,\tau\in\cS(\PHI_t)$ we have
			$$\abs{\cbc{x\in R:\sigma(x)\neq\tau(x)}}\leq k2^{-k}n.$$
\end{enumerate}
\end{theorem}

In other words, as the decimation process progresses to a point that the fraction $\theta$ of free variables satisfies~(\ref{Xeqcond}),
the set of satisfying assignments shrinks into a condensed subset of $\cbc{0,1}^{V_t}$ of tiny diameter, in contrast
to a well-spread shattered set as in \Thm~\ref{XThm_dRSB}.
Furthermore, most marginals $M_x(\PHI_t)$ are either extremely close to $0$ or extremely close to $1$.
In fact, there is a large set $R$ of variables on which all satisfying assignments
virtually agree (more precisely: any two can't disagree on more than $k2^{-k}n$
		variables in $R$).

\Paragraph{The forced phase.}
We call a variable $x$ \emph{forced} in the formula $\Phi_t$ if $\Phi_t$ has a clause that only 
contains the variable $x$ (a `unit clause').
Clearly, in any satisfying assignment $x$ must be assigned so as to satisfy this clause.

\begin{theorem}\label{XThm_forced}
There are constants $k_0,\rho_0>0$ such that for $k\geq k_0$, $\rho_0\leq\rho\leq k\ln2-2\ln k$, and
	\begin{equation}\label{Xeqforced}
		1/n\ll k\cdot\theta<\ln(\rho)(1-10/\ln\rho)
	\end{equation}
the random formula $\PHI_t$ has the following properties \whp
\begin{enumerate}
\item At least $0.99\theta n$ variables are forced.
\item The set $\cS(\PHI_t)$ is $\exp(2-\rho)/k$-condensed.
\end{enumerate}
\end{theorem}

\Paragraph{Belief Propagation.}
As mentioned earlier, the BP decimation algorithm is an attempt at implementing the decimation process
by means of an efficient algorithm.
The key issue with this is the computation of the marginals $M_{x_t}(\Phi_{t-1})$ in step~2 of the decimation process.
Indeed, the problem of computing these marginals is $\#P$-hard in the worst case.
Thus, instead of working with the `true' marginals, BP decimation uses certain numbers
	$\mu_{x_t}(\Phi_{t-1},\omega)$ that can be computed efficiently, where $\omega\geq1$ is an integer parameter.
The precise definition of the $\mu_{x_t}(\Phi_{t-1},\omega)$ can be found in Appendix~\ref{Apx_BP}
	(or~\cite{BMZ}).
Basically, they are the result of a `local' dynamic programming algorithm (`Belief Propagation')
	that depends upon the assumption of a certain correlation decay property.
For given $k,\rho$, the key hypothesis underpinning the BP decimation algorithm is 

\begin{hypothesis}\label{XHyp_BP}
For any $\eps>0$ there is $\omega=\omega(\eps,k,\rho,n)\geq1$ such that \whp\
for all $1\leq t\leq n$ we have
	$\abs{\mu_{x_t}(\PHI_{t-1},\omega)-M_{x_t}(\PHI_{t-1})}<\eps.$
\end{hypothesis}

In other words, Hypothesis~\ref{XHyp_BP} states that throughout the decimation process,
the `BP marginals' $\mu_{x_t}(\Phi_{t-1},\omega)$ are a good approximation to the true marginals $M_{x_t}(\Phi_{t-1})$.

\begin{theorem}\label{XThm_NoBP}
There exist constants $c_0,k_0,\rho_0>0$ such that for all $k\geq k_0$, and
$\rho_0\leq\rho\leq k\ln2-2\ln k$
the following is true for any integer $\omega=\omega(k,\rho,n)\geq1$.
Suppose that
	\begin{equation}\label{XeqNoBP}
	c_0\ln(\rho)<k\cdot\theta<\rho/\ln2.
	\end{equation}
Then for at least $0.99\theta n$ variables $x\in V_t$ we have
	$\mu_x(\PHI_t,\omega)\in\brk{0.49,0.51}.$
\end{theorem}

The proof 
is based on the techniques developed in~\cite{BPdec};
	the details are omitted from this extended abstract.\footnote{In the appendix we indicate how \Thm~\ref{XThm_NoBP}
			follows from the results of~\cite{BPdec}.}
Comparing \Thm~\ref{XThm_cond} with \Thm~\ref{XThm_NoBP}, we see that \whp\ for $\theta$
satisfying~(\ref{XeqNoBP}) most of the `true' marginals $M_x(\PHI_t)$
are very close to either $0$ or $1$,
whereas the `BP marginals' lie in $\brk{0.49,0.51}$.
Thus, in the regime described by~(\ref{XeqNoBP}) the BP marginals
do \emph{not} provide a good approximation to the actual marginals.

\begin{corollary}
There exist constants $c_0,k_0,\rho_0>0$ such that for all $k\geq k_0$,
	$\rho_0\leq\rho\leq k\ln2-3\ln k$
 Hypothesis~\ref{XHyp_BP} is untrue.
\end{corollary}

\Paragraph{Summary and discussion.}
Fix $k\geq k_0$ and $\rho\geq\rho_0$.
\Thm s~\ref{XThm_sym}--\ref{XThm_forced} show how the space of satisfying assignments
of $\PHI_t$ evolves as the decimation process progresses.
In the \emph{symmetric phase} $k\theta\geq\exp((1+o_\rho(1))\rho)$ where there still is a large number
of free variables, the correlations amongst the free variables are purely local (`loose variables').
As the number of free variables enters the regime
	$(1+o_\rho(1))\rho/\ln2\leq k\theta\leq\exp((1-o_\rho(1))\rho)$,
the set $\cS(\PHI_t)$ of satisfying assignments shatters into exponentially many tiny `clusters',
each of which comprises only an exponentially small fraction of all satisfying assignments.
Most satisfying assignments exhibit long-range correlations amongst the possible values that
can be assigned to the individual variables (`rigid variables').
This phenomenon goes by the name of \emph{dynamic replica symmetry breaking} in statistical mechanics~\cite{pnas}.

\enlargethispage{15mm}

While in the previous phases the set of satisfying assignments is scattered all over the Hamming
cube (as witnessed by the average Hamming distance of two satisfying assignments),
in the \emph{condensation phase} $(1-o_\rho(1))\ln\rho\leq k\theta\leq(1-o_\rho(1))\rho/\ln2$ the set of satisfying assignments has a tiny diameter.
This is mirrored by the fact that the marginals of most variables are extremely close to either $0$ or $1$.
Furthermore, in (most of) this phase the estimates of the marginals resulting from Belief Propagation are off (\Thm~\ref{XThm_NoBP}).
As part~4 of \Thm~\ref{XThm_cond} shows, the mistaken estimates of the Belief Propagation computation
would make it impossible for BP decimation to penetrate the condensation phase.
More precisely, even if BP decimation would emulate
	the decimation process perfectly up until the condensation phase commences,
	with probability $1-\exp(-\Omega(n))$ BP decimation would then assign
	at least $k2^{-k}n$ variables in the set $R$ from part~4 of \Thm~\ref{XThm_cond}
	`wrongly' (i.e., differently than they are assigned in any satisfying assignment).
In effect, BP decimation would fail to find a satisfying assignment,
	regardless of its subsequent decisions.
Finally, in the forced phase $k\theta\leq(1-o_\rho(1))\ln\rho$ there is an abundance of unit clauses
that  make it easy to read off the values of most variables.
However, getting stuck in the condensation phase, BP decimation won't reach this regime.

These results suggest that the reason for the failure of BP decimation is the existence of the condensation phase.
Intuitively, in the condensation phase the marginals are governed by genuinely global phenomena
(essentially expansion properties) that elude the inherently local  BP computation.
By contrast, it is conceivable that BP does indeed yield the correct marginals in the previous phases.
Verifying or falsifying this remains an important open problem.

\section{Related work}

\noindent{\bf The statistical mechanics perspective.}
BP/SP decimation are inspired by a generic
but highly non-rigorous analysis technique from statistical mechanics called the \emph{cavity method}~\cite{BMZ}.
This technique is primarily destined for the \emph{analysis} of phase transitions.
It is based on the (unproven) \emph{replica symmetry breaking hypothesis}, which aims to characterize
the possible types of correlations amongst the variables~\cite{pnas}.

In~\cite{BMZ,pnas} the cavity method was used to study the structure of the set $\cS(\PHI)$ of satisfying assignments
(or, more accurately, properties of the Gibbs measure) of the \emph{undecimated} random formula $\PHI$.
Thus, the results obtained in that (non-rigorous) work identify phase transitions solely in terms of the
formula density $\rho$.
On the basis of these results, it was hypothesized that (certain versions of) BP decimation
should find satisfying assignments up to $\rho\sim \ln k$ or even up to $\rho\sim k\ln2$~\cite{pnas}.
The argument given for the latter scenario in~\cite{pnas} is that the key obstacle for
BP to approximate the true marginals is \emph{condensation}.
In terms of the parameter $\rho$, the condensation threshold was (non-rigorously) estimated to occur at
	$\rho=k\ln2-3k2^{-k-1}\ln2$.
However, \cite{BPdec} shows that (the basic version of) BP decimation
fails to find satisfying assignments already for $\rho\geq\rho_0$, with $\rho_0$ a constant independent of $k$.

The explanation for this discrepancy is that~\cite{BMZ,pnas} neglect the time parameter $\theta=1-t/n$ of the decimation process.
As \Thm~\ref{XThm_cond} shows, even for \emph{fixed} $\rho\geq\rho_0$ (independent of $k$) condensation occurs as the decimation
process proceeds to $\theta$ in the regime~(\ref{Xeqcond}).
This means that decimating variables has a similar effect on the geometry of the set of satisfying assignments
as increasing the clause/variable density.
On a non-rigorous basis an analysis both in terms of the formula density $\rho$ and the
time parameter $\theta$ was carried out in~\cite{RTS}.
Thus, our results can be viewed as a rigorous version of~\cite{RTS}
	(with proofs based on completely different techniques).
In addition, \Thm~\ref{XThm_NoBP} confirms rigorously that for $\rho,\theta$ in the condensation phase, BP
does not yield the correct marginals.

The present results have no immediate bearing on the conceptually more sophisticated SP decimation algorithm.
However, we conjecture that SP undergoes a similar sequence of phase transitions
and that the algorithm will not find satisfying assignments for densities $\rho\geq\rho_0$,
with $\rho_0$ a certain constant independent of $k$.

\noindent{\bf Rigorous work.}
\Thm~\ref{XThm_dRSB} can be viewed as a generalization of the results on random $k$-SAT obtained in~\cite{Barriers}
	(which additionally deals with further problems such as random graph/hypergraph coloring).
In~\cite{Barriers} we rigorously proved a substantial part of the results hypothesized in~\cite{pnas} on shattering
and rigidity 
in terms of the clause/variable density $\rho$;
this improved prior work~\cite{SolSpace,Frozen,Daude}.
The new aspect of the present work is that we identify not only a transition for shattering/rigidity, but also for condensation
and forcing in terms of \emph{both} the density $\rho$ and the time parameter $\theta$ of the decimation process.
As explained in the previous paragraph, the time parameter is crucial to link these phase transitions
to the performance of algorithms such as BP decimation.

In particular, from \Thm~\ref{XThm_dRSB} we can recover the main result of~\cite{Barriers} on random $k$-SAT.
Namely, if $\rho\geq\ln k+2\ln\ln k+2$, then (\ref{XeqdRSB}) is satisfied even for $\theta=1$, i.e.,
the \emph{undecimated} random formula $\PHI$ has the properties 1.--4.\ stated in \Thm~\ref{XThm_dRSB}.
Technically, the present paper builds upon the methods developed in~\cite{Barriers}.
In addition, new arguments are needed to accommodate the time parameter $\theta$, to
prove the statements on the marginals of the variables
in \Thm s~\ref{XThm_sym}--\ref{XThm_cond}, and to establish the condensation phenomenon (\Thm~\ref{XThm_cond}).

The best current rigorous algorithmic results for random $k$-SAT are~\cite{BetterAlg,FrSu,HajiSorkin,KKL}.
For general $k$, the best current algorithm succeeds up to $\rho\sim\ln k$~\cite{BetterAlg}.

\section{Analyzing the decimation process}\label{Sec_planted}

In the rest of the paper, we are going to sketch the proofs of the main results.\footnote{Full proofs can be found in the appendix.}
In this section we perform some groundwork to facilitate a rigorous analysis of the decimation process.
The key problem is to get a handle on the following experiment:
\begin{description}
\item[D1.] Generate a random formula $\PHI$, conditioned on $\PHI$ being satisfiable.
\item[D2.] Run the decimation process for $t$ steps to obtain $\PHI_t$.
\item[D3.] Choose a satisfying assignment $\SIGMA_t\in\cS(\PHI_t)$ uniformly at random.
\item[D4.] The result is the pair $(\PHI_t,\SIGMA_t)$.
\end{description}
As throughout the paper we only work with densities $m/n$ where $\PHI$ is satisfiable
\whp, the conditioning in step~{\bf D1} is essentially void.
Recalling that the outcome of the decimation process is a uniformly random satisfying assignment
of $\PHI$, we see that the following experiment is equivalent to {\bf D1}--{\bf D4}:
\begin{description}
\item[U1.] Generate a random formula $\PHI$, conditioned on $\PHI$ being satisfiable.
\item[U2.] Choose $\SIGMA\in\cS(\PHI)$ uniformly at random.
\item[U3.] Substitute $\SIGMA(x_i)$ for $x_i$ for $1\leq i\leq t$ and simplify
		to obtain a formula $\PHI_t$.
\item[U4.] The result is the pair $(\PHI_t,\SIGMA_t)$, where 
			$\SIGMA_t:V_t\ra\cbc{0,1},\ x\mapsto\SIGMA(x).$
\end{description}

\begin{fact}\label{XFact_UD}
The two probability distributions induced on formula/assignment pairs by the two experiments
{\bf D1}--{\bf D4} and {\bf U1}--{\bf U4} are identical.
\end{fact}

Still, an analysis of {\bf U1}--{\bf U4} seems difficult because of {\bf U2}:
	it is unclear how to analyze (or implement) this step directly. 
Following~\cite{Barriers},
we will surmount this problem by considering yet another experiment. 

\begin{description}
\item[P1.] Choose an assignment $\SIGMA'\in\cbc{0,1}^V$ uniformly at random.
\item[P2.] Choose a formula $\PHI'$ with $m$ clauses that is satisfied by $\SIGMA'$ uniformly at random.
\item[P3.] Substitute $\SIGMA'(x_i)$ for $x_i$ for $1\leq i\leq t$ and simplify
			to obtain a formula $\PHI'_t$.
\item[P4.] The result is the pair $(\PHI'_t,\SIGMA'_{t})$, where 
			$\SIGMA'_{t}:V_t\ra\cbc{0,1},\ x\mapsto\SIGMA'(x).$
\end{description}

The experiment {\bf P1}--{\bf P4} is easy to implement and, in effect, also amenable to a rigorous analysis.
For given the assignment $\SIGMA'$, there are $(2^k-1)\bink nk$ clauses in total that evaluate to `true' under $\SIGMA'$, and
to generate $\PHI'$ we merely  choose $m$ out of these uniformly and independently.
Unfortunately, it is \emph{not} true that the experiment {\bf P1}--{\bf P4} is equivalent to {\bf U1}--{\bf U4}.
However, we will employ a result from~\cite{Barriers} that establishes a connection between these two experiments that
is strong enough to extend many results from {\bf P1}--{\bf P4} to {\bf U1}--{\bf U4}.

To state this result, observe that {\bf P1}--{\bf P4} and {\bf U1}--{\bf U4} essentially only differ in their first two steps.
Thus, let $\Lambda_k(n,m)$ denote the set of all pairs $(\Phi,\sigma)$, where $\Phi$ is a $k$-CNF on $V=\cbc{x_1,\ldots,x_n}$
with $m$ clauses, and $\sigma\in\cS(\Phi)$.
Let $\cU_k(n,m)$ denote the probability distribution induced on $\Lambda_k(n,m)$ by {\bf U1}--{\bf U2}, and let
$\cP_k(n,m)$ signify the distribution induced by {\bf P1}--{\bf P2};
this distribution is sometimes called the \emph{planted model}.

\begin{theorem}[\cite{Barriers}]\label{XCor_SATExchange}\label{Cor_SATExchange}
Suppose $k\geq4$ and $0<\rho<k\ln2-k^2/2^{k}$.
Let  $\EE\subset\Lambda_k(n,m)$. 
If
	$\pr_{\mathcal{P}_k\bc{n,m}}\brk{\EE}
		\geq1-\exp(-\rho n/2^k)$
then $\pr_{\mathcal{U}_k\bc{n,m}}\brk{\EE}=1-o(1).$
\end{theorem}

\section{Shattering, pairwise distances, and condensation}

To prove shattering and condensation, we adapt arguments from~\cite{Barriers,SolSpace,Daude} to the
situation where we have the \emph{two} parameters $\theta,\rho$ (rather than just $\rho$).
Let $(\PHI_t,\SIGMA_t)$ be the (random) outcome of the experiment {\bf U1}--{\bf U4}.
For $0\leq\alpha\leq1$ let $X_\alpha(\PHI_t,\SIGMA_t)$ denote the number of satisfying assignments $\tau\in\cS(\PHI_t)$
with Hamming distance $d(\SIGMA_t,\tau)=\alpha\theta n$.
To establish the `shattering' part of \Thm~\ref{XThm_dRSB}, we are going to prove the following

\begin{Claim}\label{XClaim_shattering}
Under the assumptions of \Thm~\ref{XThm_dRSB} there exist
$a_1<a_2<0.49$, $a_3>0$ depending only on $k,\rho$ such that \whp\ 
we have
	\begin{eqnarray}\label{Xeqsh1}
	X_\alpha(\PHI_t,\SIGMA_t)&=&0\quad\mbox{ for all $a_1<\alpha<a_2$, and}\\
	\max_{\alpha\leq0.49}X_\alpha(\PHI_t,\SIGMA_t)&<&\exp(-a_3n)\cdot\abs{\cS(\PHI_t)}.	\label{Xeqsh2}
	\end{eqnarray}
\end{Claim}

Claim~\ref{XClaim_shattering} implies that for the outcome $\PHI_t$ of the first $t$ steps
of the decimation process the set $\cS(\PHI_t)$ shatters \whp\
For by Fact~\ref{XFact_UD} Claim~\ref{XClaim_shattering} implies that \whp\
almost all $\sigma_t\in\cS(\PHI_t)$ are such that
(\ref{Xeqsh1}) and~(\ref{Xeqsh2}) hold.
Choose any such $\sigma_{t,1}\in\cS(\PHI_t)$ and let $R_1=\cbc{\tau\in\cS(\PHI_t):d(\tau,\sigma_{t,1})\leq a_1n}$.
Then, choose $\sigma_{t,2}\in\cS(\PHI_t)\setminus R_1$
	satisfying
	(\ref{Xeqsh1}) and~(\ref{Xeqsh2}), let 
$R_2=\cbc{\tau\in\cS(\PHI_t)\setminus R_1:d(\tau,\sigma_{t,2})\leq a_1n}$, and proceed inductively until
all remaining satisfying assignments violate either~(\ref{Xeqsh1}) or (\ref{Xeqsh2}).
Let $R_1,\ldots,R_N$ be the classes constructed in this way and let $R_0=\cS(\PHI_t)\setminus\bigcup_{i=1}^NR_i$.
An additional (simple) argument is needed to show that $|R_0|\leq\exp(-\Omega(n))|\cS(\PHI_t)|$ \whp\
The decomposition $R_0,\ldots,R_N$ witnesses that $\cS(\PHI_t)$ shatters.

With respect to pairwise distances of satisfying assignments, (\ref{Xeqsh2}) implies that \whp\ only
an exponentially small fraction of all satisfying assignments of $\PHI_t$ lies within distance $\leq0.49\theta n$
of $\SIGMA_t$.
It is not difficult to derive the statement made in \Thm~\ref{XThm_dRSB} on the average pairwise distance
from this.
In addition, the fact that the average pairwise distance of satisfying assignments is $\geq0.49\theta n$ \whp\
implies in combination with a double counting argument the claim about the marginals $M_x(\PHI_t)$
in \Thm s~\ref{XThm_sym} and~\ref{XThm_dRSB}.

To establish Claim~\ref{XClaim_shattering} we will work with the experiment {\bf P1}--{\bf P4}
and use \Thm~\ref{XCor_SATExchange} to transfer the result to the experiment {\bf U1}--{\bf U4}.
Thus, let $(\PHI_t',\SIGMA_t')$ be the (random) outcome of experiment {\bf P1}--{\bf P4},
and assume that $k,\rho,\theta$ are as in \Thm~\ref{XThm_dRSB}.
To prove~(\ref{Xeqsh1}) we need to bound $X_\alpha(\PHI_t',\SIGMA_t')$ from above, for which we use
the `first moment method'.
Indeed, by standard arguments (similar to those used in~\cite{SolSpace}) the expectation of $X_\alpha(\PHI_t',\SIGMA_t')$ satisfies
	$\frac1n\ln\Erw X_\alpha(\PHI_t',\SIGMA_t')\leq\psi(\alpha)$, with
	$$\psi(\alpha)=
		-\alpha\theta\ln\alpha-(1-\alpha)\theta\ln(1-\alpha)+
			\frac{2^k\rho}k\ln\bc{1-\frac{1-(1-\alpha\theta)^k}{2^k-1}}.$$
Thus, in order to prove that $\max_{a_1<\alpha<a_2}X_\alpha(\PHI_t',\SIGMA_t')=0$ \whp\ we would just have to prove that
$\max_{a_1<\alpha<a_2}\psi(\alpha)<0$ (so that Markov's inequality implies that $X_{\alpha}=0$ \whp).
But as our goal is to prove a result about the $X_\alpha(\PHI_t,\SIGMA_t)$ (i.e., the experiment {\bf U1}--{\bf U4}),
we need to prove a slightly stronger bound, namely
	\begin{equation}\label{Xeqpsialphastrong}
	\max_{a_1<\alpha<a_2}\psi(\alpha)<-\rho/2^k.
	\end{equation}
Then Markov's inequality and \Thm~\ref{XCor_SATExchange} imply the first part of Claim~\ref{XClaim_shattering}.
Via elementary calculus, one can show that~(\ref{Xeqpsialphastrong}) holds with
$a_1=\exp(2-\rho)-\eps$ and $a_2=\exp(2-\rho)+\eps$ for a sufficiently small $\eps>0$.

To prove~(\ref{Xeqsh2}) we bound $\Erw X_\alpha$ from above by a similar first moment argument.
But in addition, we need a lower bound on $\abs{\cS(\PHI_t)}$.
To derive this lower bound, we need

\begin{theorem}[\cite{SolSpace}]\label{XThm_count}\label{Thm_count}
Assume $k\geq 4$ and $\rho\leq k\ln2-k^2/2^k$.
Then \whp\ 
	$\frac1n\ln\abs{\cS(\PHI)}\geq\ln2+2^k\frac\rho k\ln(1-2^{-k})-0.99\rho/2^k.$
\end{theorem}
In combination with a double counting argument, \Thm~\ref{XThm_count} implies the following lower bound
on $\cS(\PHI_t)$, which entails the second part of Claim~\ref{XClaim_shattering}.

\begin{corollary}\label{XCor_count}
Let $(\PHI_t,\SIGMA_t)$ be the outcome of {\bf U1}--{\bf U4}.
Let $1\leq t\leq n$.
Then \whp\
	$\frac1n\ln\abs{\cS(\PHI_t)}\geq\theta\ln2+2^k\frac{\rho}k\ln(1-2^{-k})-\frac{\rho}{2^k}.$
\end{corollary}

The proof of the `condensation' part of \Thm~\ref{XThm_cond} is based on similar arguments.
Basically, to show condensation we need to prove that $\max_{\alpha>a_4}X_\alpha(\PHI_t,\SIGMA_t)<0$
with probability $1-\exp(-\Omega(n))$, where we let $a_4=\exp(2-\rho)$.
This  is done via the first moment method and  boils down to proving that $\psi(\alpha)<-\rho/2^k$
for all $\alpha>a_4$.

\section{Rigid variables}\label{Sec_rigidOutline}

Assume that $k,\rho,\theta$ satisfy the assumptions of \Thm~\ref{XThm_dRSB}.
Let $(\PHI_t,\SIGMA_t)$ be the (random) outcome of {\bf U1}--{\bf U4}.
Our goal is to show that \whp\ most variables $x\in V_t$ are rigid.

What is the basic obstacle that makes it difficult to `flip' the value of $x$?
Observe that we can simply assign $x$ the opposite value $1-\SIGMA_t(x)$, unless
$\PHI_t$ has a clause $\cC$ in which either $x$ or $\bar x$ is the \emph{only} literal that is true under $\SIGMA_t$.
If there is such a clause, we say that $x$ \emph{supports} $\cC$.
But even if $x$ supports a clause $\cC$ it might be easy to flip.
For instance, if $\cC$ features some variable $y\neq x$ that does not support a clause,
then we could just flip both $x,y$ simultaneously.
Thus, to establish the existence of $\Omega(n)$-rigid variables we need to analyze the distribution of the number of clauses that a variable supports,
the probability that these clauses only consists of variables that support further clauses,
the probability that the same is true of those clauses, etc.

This analysis can be performed fairly neatly for the outcome $(\PHI_t',\SIGMA_t')$ of the experiment {\bf P1}--{\bf P4}.
Let us sketch how this works, and why rigidity occurs at $k\theta=\exp((1+o(1))\rho)$ (cf.~(\ref{XeqdRSB})).
For a variable $x\in V_t$ we let $S_x$ be the number of clauses supported by $x$.
Given the assignment $\SIGMA'$ chosen in step {\bf P1}, there are a total of $\bink{n-1}{k-1}$ possible
clauses that $x$ supports.
Since in step {\bf P2} we include $m$ out of the $(2^k-1)\bink nk$ possible clauses satisfied under $\SIGMA'$ uniformly
and independently, we get
	$$\Erw\brk{S_x}=m\bink{n-1}{k-1}\bc{(2^k-1)\bink nk}^{-1}=\rho/(1-2^{-k})\geq\rho.$$
In fact, $S_x$ is binomially distributed.
Hence, $\pr\brk{S_x=0}\leq\exp(-\rho)$.
Thus, the \emph{expected} number of variables $x\in V_t$ with $S_x=0$ is $\leq\theta n\exp(-\rho)$.
Furthermore, if we condition on $S_x=j\geq1$, then the actual \emph{clauses} $\cC_1,\ldots,\cC_j$
supported by $x$ are just independently uniformly distributed over the set of all $\bink{n-1}{k-1}$ possible clauses that $x$ supports.
Therefore, the \emph{expected} number of variables $y\in V_t$ with $S_y=0$ occurring in one of these clauses $\cC_i$
is $(1+o(1))(k-1)\cdot \theta\exp(-\rho)\leq k\theta\exp(-\rho)$.
Hence, if $\theta$ is as in~(\ref{XeqdRSB}), then this number is
	$\leq\exp(-2)/\rho$, i.e., `small' for $\rho\geq\rho_0$ sufficiently big.
Thus, we would expect that \emph{most} clauses supported by $x$ indeed consist exclusively of variables that support other clauses.
This heuristic argument shows that for $\theta$ as in~(\ref{XeqdRSB}) we can plausibly expect most variables to be rigid.

Let us now indicate how this argument can be carried out in detail.
Analyzing the distribution of the variables $S_x$ in the experiment {\bf P1}--{\bf P4} and extending
the result to the experiment {\bf U1}--{\bf U4}
via \Thm~\ref{XCor_SATExchange}, and setting $\zeta=\rho^2/\exp(\rho)$, we obtain the following.

\begin{proposition}\label{XProp_Supp}
Suppose that $k,\rho,\theta$ satisfy the assumptions of \Thm~\ref{XThm_dRSB}.
Then \whp\ in a random pair $(\PHI_t,\SIGMA_t)$ generated by the experiment {\bf U1}--{\bf U4} no more than
$2\zeta\theta n$ variables in $V_t$ support fewer than three clauses,
\end{proposition}

To establish rigidity, we need to show that most variables support clauses in which only variables
occur that support other clauses.
To express this, we say that $S\subset V_t$ is $t$-{\em self-contained}
if each $x\in S$ supports at least two clauses of $\PHI_t$ that contain
variables from $S$ only.
From \Prop~\ref{XProp_Supp} we can derive the following.

\begin{proposition}\label{XProp_RandPoisson}
Suppose that $k,\rho,\theta$ satisfy the assumptions of \Thm~\ref{XThm_dRSB}.
The outcome $(\PHI_t,\SIGMA_t)$ of {\bf U1}--{\bf U4} has 
a $t$-self-contained set of size $(1-3\zeta)\theta n$ \whp
\end{proposition}

Suppose that $(\PHI_t,\SIGMA_t)$ has a self-contained set $S$ of size $(1-3\zeta)\theta n$.
To flip the value of a variable $x\in S$ we need to also flip one other variable from each of the (at least two) clauses that $x$ supports
and that consist of variables from $S$ only.
As each of these two variables, in turn, supports at least two clauses comprised of variables from $S$ only, we need to also flip further variables in those.
But these variables are again contained in $S$.
This suggests that attempting to flip $x$ will entail an avalanche of further flips.
Indeed, the expansion properties of the random formula $\PHI_t$ imply the following.

\begin{proposition}\label{XProp_frozen}
Suppose that $k,\rho,\theta$ satisfy the assumptions of \Thm~\ref{XThm_dRSB}.
There is $\chi=\chi(k,\rho)>0$ such that
the outcome $(\PHI_t,\SIGMA_t)$ of {\bf U1}--{\bf U4} has the following property \whp:
all variables that are contained in a $t$-self-contained set  are $\chi n$-rigid.
\end{proposition}

\Prop s~\ref{XProp_RandPoisson} and~\ref{XProp_frozen} directly imply part~1 of \Thm~\ref{XThm_dRSB}.
Self-contained sets also play a key role in the proof of \Thm~\ref{XThm_cond}.
\Prop s~\ref{XProp_RandPoisson} and~\ref{XProp_frozen} can be extended to the regime of $\theta$
as in \Thm~\ref{XThm_cond}, and the set $R$ in part~4 of that theorem is simply a $t$-self-contained set.
Expansion properties of the random formula together with the bound on the diameter of
the set $\cS(\PHI_t)$ of satisfying assignments from part~2 of  \Thm~\ref{XThm_cond}
imply that there are no two satisfying assignments that disagree on more than $kn/2^k$ variables from $R$.
In combination with a double-counting argument, this implies the statement on the marginals in part~3
of  \Thm~\ref{XThm_cond}.
Finally, the claim about forced variables in \Thm~\ref{XThm_forced} can be proved via a similar
(but simpler) argument as sketched in this section.

\newpage
%
%
%

\begin{appendix}
\noindent{\Large\bf Appendix}

\bigskip
\noindent
Appendix~\ref{Apx_BP} contains a discussion of Belief Propagation.
The remaining appendices contain the full proofs of the results stated in \Sec~\ref{Sec_results}.
Appendix~\ref{Apx_overview} gives an overview of how the proofs are organized.

\section{Detailed description of BP decimation}\label{Apx_BP}

The BP decimation algorithm can be viewed as an attempt at implementing
the decimation process (Experiment~\ref{Exp_dec}).
As mentioned earlier, the key issue with this is the computation (or approximation) of the
the marginals $M_x(\Phi_{t-1})$.
BP decimation basically tries to approximate these marginals by means of a `local' computation.

For clearly, the marginals $M_x(\Phi_{t-1})$ are influenced by `local' effects.
For instance, if $x$ occurs in a unit clause $a$ of $\Phi_{t-1}$, i.e., a clause of length one,
 then $x$ \emph{must} be assigned so as to satisfy $a$.
Hence, if $x$ appears in $a$ positively, then $M_x(\Phi_{t-1})=1$, and otherwise $M_x(\Phi_{t-1})=0$.
Similarly, if $x$ occurs \emph{only} positively in $\Phi_{t-1}$, then $M_x(\Phi_{t-1})\geq1/2$.
More intricately, if $x$ occurs in a clause $a$ that contains another variable $y$ that appears is a unit clause $b$,
then this will affect the marginal of $x$.

The key hypothesis underlying {\tt BPdec} is that in random formulas
such local effects \emph{determine} the marginals $M_x(\Phi_{t-1})$ asymptotically.
To define `local' precisely, we need a metric on the variables/clauses.
This metric is the one induced by the \emph{factor graph} $G=G(\Phi_{t-1})$ of $\Phi_{t-1}$,
which is a bipartite graph whose vertices are the variables $V_{t-1}=\cbc{x_t,\ldots,x_n}$ and the clauses of $\Phi_{t-1}$.
Each clause is adjacent to the variables that occur in it.
For an integer $\omega\geq1$ let $N^\omega(x_t)$ signify the set of all vertices of $G$ that have distance at
most $2\omega$ from $x_t$.
Then the induced subgraph $G\brk{N^{\omega}(x_t)}$ corresponds to the sub-formula
of $\Phi_{t-1}$ obtained by removing all clauses and variables at distance more than $2\omega$ from $x_t$.
Note that all vertices at distance precisely $2\omega$ are variables, so that any satisfying assignment of
$\Phi$ induces a satisfying assignment of the sub-formula.
Let us denote by $M_{x_t}(\Phi_{t-1},\omega)$ the marginal probability that $x_t$ takes the value $1$
in a random satisfying assignment of this sub-formula.

Of course, for a worst-case formula $\Phi$ the `local' marginals $M_{x_t}(\Phi_{t-1},\omega)$ may
be just as difficult to compute as the overall marginals $M_{x_t}(\Phi_{t-1})$ themselves.
Therefore,  BP decimation employs an efficient dynamic programming heuristic called \emph{Belief Propagation} (`BP'),
which yields certain values $\mu_{x_t}(\Phi_{t-1},\omega)\in\brk{0,1}$;
we will state this heuristic below.
If the induced subgraph $G\brk{N^{\omega}(x_t)}$ is a tree, then indeed
$\mu_{x_t}(\Phi_{t-1},\omega)=M_{x_t}(\Phi_{t-1},\omega)$.
Moreover, standard arguments show that in a random formula $\PHI$ actually
$G\brk{N^{\omega}(x_t)}$ is a tree \whp\ so long as $\omega=o(\ln n)$.
Of course, more generally, in order to obtain an efficient algorithm it would be sufficient
for the BP outcomes $\mu_{x_t}(\Phi_{t-1},\omega)$ to approximate the true overall marginals $M_{x_t}(\Phi_{t-1})$ well for some
polynomially computable and polynomially bounded function $\omega=\omega(n)\geq1$.

To define the numbers $\mu_{x_t}(\Phi_{t-1},\omega)$ formally,
we need to define Belief Propagation for $k$-SAT. 
To this end, let $N(v)$ denote the neighborhood of a vertex $v$ of the factor graph $G(\Phi_{t-1})$.
For a variable $x\in V_t$ and a clause $a\in N(x)$ we will denote the ordered pair $(x,a)$ by $x\ra a$.
Similarly, $a\ra x$ stands for the pair $(a,x)$.
Furthermore, we let $\sign(x,a)=1$ if $x$ occurs in $a$ positively, and  $\sign(x,a)=-1$ otherwise.

The \emph{message space} $M(\Phi_{t-1})$ is the set of all tuples
	$(\mu_{x\ra a}(\zeta))_{x\in V_t,\,a\in N(x),\,\zeta\in\{0,1\}}$ such that $\mu_{x\ra a}(\zeta)\in\brk{0,1}$
and $\mu_{x\ra a}(0)+\mu_{x\ra a}(1)=1$ for all $x,a,\zeta$.
For $\mu\in M(\Phi)$ we define 
	$\mu_{a\ra x}(\zeta)=1$ if $\zeta=(1+\sign(x,a))/2$,
and 
	\begin{equation}\label{eqBPai}
	\mu_{a\ra x}(\zeta)=
			1-\hspace{-4mm}\prod_{y\in N(a)\setminus\cbc x}\hspace{-4mm}
				\mu_{y\ra a}\bc{\frac{1-\sign(y,a)}2}
	\end{equation}
otherwise.
Furthermore, we define the {\em belief propagation operator} $\BP$ 
as follows:
for any $\mu\in M(\Phi_{t-1})$ we define $\BP(\mu)\in M(\Phi_{t-1})$ by letting 
	\begin{eqnarray}
	(\BP(\mu))_{x\ra a}(\zeta)&=&
	\frac{\displaystyle\prod_{b\in N(x)\setminus\cbc a}\mu_{b\ra x}(\zeta)}
			{\displaystyle\prod_{b\in N(x)\setminus\cbc a}\mu_{b\ra x}(0)+
			\prod_{b\in N(x)\setminus\cbc a}\mu_{b\ra x}(1)}
			\label{eqBPia}
	\end{eqnarray}
unless the denominator equals zero, in which case $(\BP(\mu))_{x\ra a}(\zeta)=\frac12$.

Finally, the values $\mu_x(\Phi_t,\omega)$ are defined as follows.
Let $\mu\brk0=\frac12\cdot\vecone\in M(\Phi_{t-1})$ be the vector with all entries equal to $\frac12$.
Moreover, define inductively $\mu\brk{\ell}=\BP(\mu\brk{\ell-1})$ for $1\leq\ell\leq\omega$.
Then
	\begin{eqnarray}
	\mu_x(\Phi_{t-1},\omega)&=&
		\frac{\displaystyle\prod_{b\in N(x)}\mu_{b\ra x}(1)\brk\omega}
			{\displaystyle\prod_{b\in N(x)}\mu_{b\ra x}(0)\brk\omega+
			\prod_{b\in N(x)}\mu_{b\ra x}(1)\brk\omega}
		\label{eqBPmarginal}
	\end{eqnarray}
for any $x\in V_t$,
unless the denominator is zero, in which case we set $\mu_x(\Phi_{t-1},\omega)=\frac12$.

\begin{figure}
\noindent{
\begin{algorithm}\label{Alg_BPdec}\upshape\texttt{BPdec$(\Phi)$}\\\sloppy
\emph{Input:} A $k$-CNF $\Phi$ on $V=\cbc{x_1,\ldots,x_n}$.\\
\emph{Output:} An assignment $\sigma:V\rightarrow\cbc{0,1}$.
\vspace{-2mm}
\begin{tabbing}
mmm\=mm\=mm\=mm\=mm\=mm\=mm\=mm\=mm\=\kill
{\algstyle 0.}	\> \parbox[t]{30em}{\algstyle
	Let $\Phi_0=\Phi$.}\\
{\algstyle 1.}	\> \parbox[t]{30em}{\algstyle
	For $t=1,\ldots,n$ do}\\
{\algstyle 2.}	\> \> \parbox[t]{28em}{\algstyle
		Use BP to compute $\mu_{x_t}(\Phi_{t-1},\omega)$.}\\
{\algstyle 3.}	\> \> \parbox[t]{28em}{\algstyle
		Assign $\sigma(x_t)=1$ with probability $\mu_{x_t}(\Phi_{t-1},\omega)$, and let $\sigma(x_t)=0$ otherwise.}\\
{\algstyle 4.}	\> \> \parbox[t]{28em}{\algstyle
		Obtain the formula $\Phi_t$ from $\Phi_{t-1}$ by substituting the value $\sigma(x_t)$ for $x_t$ and simplifying.}\\
{\algstyle 5.}	\> \parbox[t]{30em}{\algstyle
	Return the assignment $\sigma$.}
\end{tabbing}
\end{algorithm}}
\caption{The BP decimation algorithm.}\label{Fig_BPdec}
\end{figure}

The intuition here is that the $\mu_{x\ra a}(\zeta)$ are `messages' from a variable $x$ to the clauses $a$ in which
$x$ occurs, indicating how likely $x$ were to take the value $\zeta$ if clause $a$ were removed from the formula.
Based on these, (\ref{eqBPai}) yields messages $\mu_{a\ra x}(\zeta)$ from clauses $a$ to variables $x$, indicating
the probability that $a$ is satisfied if $x$ takes the value $\zeta$ and all other variables $y\in N(a)\setminus\cbc x$ are assigned
independently to either value $\xi\in\cbc{0,1}$ with probability $\mu_{y\ra a}(\xi)$.
The BP operator~(\ref{eqBPia}) then uses these messages $\mu_{a\ra x}$ in order to `update' the messages
from variables to clauses.
More precisely, for each $x$ and $a\in N(x)$ the new messages $(\BP(\mu))_{x\ra a}(\zeta)$ are computed
under the hypothesis that all other clauses $b\in N(x)\setminus\cbc a$ are satisfied with probabilities $\mu_{b\ra x}(\zeta)$ independently
if $x$ takes the value $\zeta$.
Finally, the difference between~(\ref{eqBPia}) and~(\ref{eqBPmarginal}) is that the latter product runs over \emph{all} clauses $b\in N(x)$.
An inductive proof shows that, if for a variable $x$ the subgraph $G\brk{N^\omega(x)}$ of the factor graph is a tree, then
in fact $\mu_x(\Phi_t,\omega)=M_x(\Phi_t,\omega)$~\cite{BMZ}.
 Figure~\ref{Fig_BPdec} shows the BP decimation algorithm.

\section{Overview}\label{Apx_overview}

In \Sec~\ref{Sec_results} we described the main results of this paper
arranged according to the various phases that the decimation process passes through.
But to prove these results, it is necessary to proceed in a different order.
To facilitate this, we will state the main results in the order in which the proofs proceed.
We begin with the statements on the loose/rigid/forced variables.

\begin{theorem}\label{Thm_vars}
There exist constants $k_0,\rho_0>0$ such that for all $k\geq k_0$ and $\rho_0\leq\rho\leq k\ln2-2\ln k$ the following three statements hold
for a random pair $(\PHI_t,\sigma_t)$ chosen from the experiment {\bf U1}--{\bf U4} \whp
\begin{enumerate}
\item If
		$k\theta>\exp\brk{\rho\bc{1+\frac{\ln\ln\rho}{\rho}+\frac{10}{\rho}}},$
	then 
		at least $0.99\theta n$ variables $x\in V_t$ are loose \whp
\item If
	$1<k\theta<\exp\brk{\rho\bc{1-\frac{3\ln\rho}{\rho}}},$ then 
		at least $\rho^3\exp(-\rho)\theta n$ 
				variables $x\in V_t$ are $\Omega( n)$-rigid \whp
\item If
		$\ln(n)/n<\theta<(\ln(\rho)-10)/k,$ then at least $0.99\theta n$ variables
	are forced \whp
\end{enumerate}
\end{theorem}

The second type of statement concerns the global structure of the set of satisfying assignments,
summarized in the following theorem.

\begin{theorem}\label{Thm_shape}
There exist constants $k_0,\rho_0>0$ such that for all $k\geq k_0$, 
and $\rho_0\leq\rho\leq k\ln2-2\ln k$ the following three statements hold.
\begin{enumerate}
\item If 
		$$\frac{\rho}{\ln2}(1+\rho^{-2}+2^{2-k})\leq k\theta\leq\exp\brk{\rho\bc{1-\frac{\ln\rho}\rho-\frac2\rho}}$$
		then $\cS(\PHI_t)$ is $(\exp(2-\rho)-\eps,\exp(2-\rho)+\eps)$-shattered \whp\
		for some $\eps=\eps(k,\rho)>0$.
\item If $\theta<(\rho-1/\rho)/\bc{k\ln 2}$, then $\cS(\PHI_t)$ is $\exp(2-\rho)$-condensed \whp
\item If $\theta>\rho(1+2/\rho^2)/(k\ln 2)$, then the average distance between two random elements
		of $\cS(\PHI_t)$ is at least $0.49\theta n$ \whp
\end{enumerate}
\end{theorem}

The next theorem contains the statements about the marginals of the truth values of individual variables.

\begin{theorem}\label{Thm_marginals}
There exist constants $k_0,\rho_0>0$ such that for all $k\geq k_0$, 
and $\rho_0\leq\rho\leq k\ln2-2\ln k$ the following two statements hold.
\begin{enumerate}
\item If $\theta\geq\frac\rho{k\ln 2}(1+1/\rho^2+k/2^{k-2})$, then \whp\ for at least $\theta n/3$ variables $x\in V_t$ we have
		$$M_x(\PHI_t)\in\brk{0.01,0.99}.$$
\item If $\ln(n)/n<\theta<\rho(1-1/\rho^2)/(k\ln 2)$, then \whp\
		for all but $\exp(-\rho)\theta n$ variables $x\in V_t$ we have
		$$M_x(\PHI_t)\in\brk{0,2^{-k/2}}\cup\brk{1-2^{-k/2},1}
				.$$
\end{enumerate}
\end{theorem}

\Thm s~\ref{XThm_sym}--\ref{XThm_forced} follow directly from \Thm s~\ref{Thm_vars}--\ref{Thm_marginals}
by reordering the individual statements according to the phases they appear in,
apart from part~4 of \Thm~\ref{XThm_cond}, whose proof is given in Appendix~\ref{Sec_Marg01}.
After stating some preliminaries in Appendix~\ref{Apx_pre},
we will prove \Thm~\ref{Thm_vars} in Appendix~\ref{Sec_vars}.
Then, in Appendix~\ref{Sec_shattering} we will prove \Thm~\ref{Thm_shape}.
Further, Appendix~\ref{Sec_marg} contains the proof of \Thm~\ref{Thm_marginals}.
Finally, in Appendix~\ref{Sec_BPproofs} we prove \Thm~\ref{XThm_NoBP}.

\section{Preliminaries}\label{Apx_pre}

Recall that $V_t=\cbc{x_{t+1},\ldots,x_n}$.
In addition, we let $L_t=\cbc{x_{t+1},\bar x_{t+1},\ldots,x_n,\bar x_n}$.
For a literal $l$ let $\abs l$ be the underlying variable.
For a formula $\Phi$ on $V=\cbc{x_1,\ldots,x_n}$, an assignment $\sigma\in\cbc{0,1}^V$,
and $1\leq t\leq n$ we let
	$\Phi_{t,\sigma}$ denote the formula obtained by substituting $\sigma(x_s)$ for $x_s$ for all
	$1\leq s\leq t$ and simplifying.

We need the following \emph{Chernoff bound}
on the tails of a binomially distributed random variable $X$ with mean $\lambda$%
: for any $t>0$
\begin{eqnarray}\label{eqChernoff}
\pr(X\geq\lambda+t)\leq\exp\bc{-t\cdot\varphi(t/\lambda)}
   &\textrm{ and }&\pr(X\leq\lambda-t)\leq\exp\bc{-t\cdot\varphi(-t/\lambda)}.
\end{eqnarray}
where
	\begin{equation}\label{eqvarphi}
	\varphi(x)=(1+x)\ln(1+x)-x.
	\end{equation}

We will need the following consequence of \Thm~\ref{Thm_count}
	(cf.\ \Cor~\ref{XCor_count} in the main part of the paper).

\begin{corollary}\label{Cor_count}
Let $1\leq t\leq n$.
Let $(\Phi_t,\sigma_t)$ be a pair chosen from the experiment {\bf U1}--{\bf U4}.
Then \whp\
	\begin{equation}\label{eqCor_count}
	\frac1n\ln\abs{\cS(\Phi_t)}\geq(1-t/n)\ln2+r\ln(1-2^{-k})-\frac{kr}{4^k}.
	\end{equation}
\end{corollary}
\begin{proof}
Let $\Phi$ be a formula such that 
	$\frac1n\ln\abs{\cS(\Phi)}\geq\ln2+r\ln(1-2^{-k})-kr/4^k$.
By \Thm~\ref{Thm_count} the random formula $\PHI$ has this property \whp\
Thus, it suffices to show that for a random $\sigma\in\cS(\Phi)$ the bound~(\ref{eqCor_count}) holds \whp\
To this end, let $\cI=\cbc{0,1}^t$.
Moreover, for each $\sigma\in\cbc{0,1}^n$ let
	$\sigma|_t$ be the vector $(\sigma(x_1),\ldots,\sigma(x_t))\in\cI$.
For each $\sigma_*\in\cI$ let $Z(\sigma_*)$ be the number of assignments $\sigma\in\cS(\Phi)$ such
that $\sigma|_t=\sigma_*$.
If $\sigma\in\cS(\Phi)$ is chosen uniformly at random, then for any $\sigma_*\in\cI$ we have
	\begin{eqnarray*}
	\pr\brk{\sigma|_t=\sigma_*}&=&Z(\sigma_*)/Z,\mbox{ where }Z=\sum_{\tau\in\cI}Z(\tau)=\abs{\cS(\Phi)}.
	\end{eqnarray*}
Let $\xi>0$ be a sufficiently small number and let
	$$q=\pr\brk{Z(\sigma|_t)<\exp(-t\ln2-\xi n)\cdot Z},$$
where $\sigma\in\cS(\Phi)$ is chosen uniformly at random.
Then
	\begin{eqnarray*}
	q&=&\sum_{\sigma_*\in\cI:Z(\sigma_*)\leq\frac{Z}{\exp(\xi n+t\ln2)}}Z(\sigma_*)/Z
		\leq\frac{2^t}Z\cdot\frac{Z}{\exp(\xi n+t\ln2)}\leq\exp(-\xi n),
	\end{eqnarray*}
whence the assertion follows.
\qed\end{proof}

In \Sec~\ref{Sec_planted} we introduced the experiment {\bf P1}--{\bf P4}, which led to the planed model $\cP_k(n,m)$.
In addition, we need the following variant of the planted model.

\begin{description}
\item[P1'.] Choose an assignment $\SIGMA'\in\cbc{0,1}^V$ uniformly at random.
\item[P2'.] Choose a formula $\PHI'$ 
			by including each of the $(2^k-1)\bink{n}k$ possible clauses that are satisfied under $\SIGMA'$
			with probability $p=m/((2^k-1)\bink{n}k)$ independently.
\item[P3'.] Substitute $\SIGMA'(x_i)$ for $x_i$ for $1\leq i\leq t$ and simplify
			to obtain a formula $\PHI'_t$.
\item[P4'.] The result is the pair $(\PHI'_t,\SIGMA'_{t})$, where 
			$\SIGMA'_{t}:V_t\ra\cbc{0,1},\ x\mapsto\SIGMA'(x).$
\end{description}

Steps {\bf P1'}--{\bf P2'} of this experiment induce a probability distribution $\cP_k'(n,m)$
on formula/assignment pairs.
The following corollary establishes a connection between this distribution and the distribution $\cU_k(n,m)$.

\begin{corollary}[\cite{Barriers}]\label{Cor_Pnm'}
Suppose that $k\geq4$ and $0<r<2^k\ln2-k$.
Let  $\EE$ be any property of formula/assignment pairs.
If
	$\pr_{\mathcal{P}_k'\bc{n,m}}\brk{\EE}
		\geq1-\exp(-krn/4^k)$
then $\pr_{\mathcal{U}_k\bc{n,m}}\brk{\EE}=1-o(1).$
\end{corollary}

We will need the following elementary observation about the distribution $\cP_k'\bc{n,m}$.

\begin{lemma}\label{Lemma_SuppBin}
Let $(\Phi,\sigma)$ be a pair chosen from the distribution $\cP_k'\bc{n,m}$.
\begin{enumerate}
\item For each literal $l$ that is true under $\sigma$ the
		number of clauses supported by $l$ is binomially distributed 
			$\Bin(\frac{k}n\cdot\bink nk,m/((2^k-1)\bink{n}k))$.
\item For any integer $D$ the number of literals $l$ that support fewer than $D$ clauses
		is binomially distributed with mean
				\begin{equation}\label{eqSuppBin}
				n\cdot\pr\brk{\Bin\bc{\frac{k}n\cdot\bink nk,\frac m{(2^k-1)\bink{n}k}}<D}.
				\end{equation}
\end{enumerate}
\end{lemma}
\begin{proof}
Without loss of generality we may condition on $\sigma$ assigning the value true to all variables.
For any variable $x$ let $\cS_x$ be the set of all possible clauses in which $x$ is the only positive literal.
Then $\abs{\cS_x}=k\bink{n-1}{k-1}=\frac{k}n\cdot\bink{n}k$.
(First choose one of the $k$ slots where to place $x$, then choose the $k-1$ other variables occurring in the clause;
the signs are prescribed by $x$ being the unique positive literal.)
Moreover, let $S_x$ be the number of clauses from $\cS_x$ that actually appear in the random formula $F$.
As each of the clauses in $\cS_x$ is included in $F$ with probability $p=m/((2^k-1)\bink{n}k)$ independently,
$S_x$ has a binomial distribution $\Bin(\frac{k}n\cdot\bink nk,p)$.
This establishes 1.

Since for any two variables $x,y$ we have $\cS_x\cap\cS_y=\emptyset$, 
the random variables $S_x$ are mutually independent for all variables $x$.
Therefore, the number
	$S=\sum_x 1_{\cbc{S_x<D}}$
of variables supporting fewer than $D$ clauses in $F$ is binomially distributed as well.
\qed\end{proof}

There is a natural way to associate a bipartite graph with a $k$-CNF $\Phi$, known as the {\em factor graph}.
Its vertices are the variables and the clauses of $\Phi$, and each clause is adjacent to all the variables it contains.
For a variable $x$ we let $N_{3}(x)$ be the subgraph of that
is spanned by all vertices at distance at most $3$ from $x$.
A variable $x$ is {\em tame} if $N_{3}(x)$  is acyclic and contains no more than $\ln(n)$ variables.
The following is a well-known fact about random $k$-CNFs.

\begin{proposition}\label{Prop_tame}
Suppose that $k\geq3$ and $0<r\leq2^k\ln2$.
\Whp\ all but $o(n)$ variables are tame in $\PHI$.
\end{proposition}

Finally, the following lemma expresses an elementary `expansion property' of the random formula $\PHI$.

\begin{lemma}\label{Lemma_exp_small}
There is a number $\chi=\chi(k)>0$ such that for all $0<r\leq 2^k$
the random formula $\PHI$ has the following property \whp
	\begin{equation}\label{eq_exp_small}
	\parbox{10cm}{
	There is no set $Q$ of $1\leq |Q|\leq\chi n$ variables
	such that the number of clauses containing at least two variables from $Q$
	is at least $2|Q|$.}
	\end{equation}
\end{lemma}
\begin{proof}
We use a first moment argument.
Let $1\leq q\leq\chi n$ and let $Q_0=\{x_1,\ldots,x_q\}$ be a fixed set of size $q$.
For any set $Q$ we let $Y(Q)$ be the number of clauses containing
at least two variables from $Q$.
Moreover, let $X_q$ be the number of sets $Q$ of size $q$ such that $Y(Q)\geq2q$.
Since the distribution $\form_k(n,m)$ is symmetric with respect to permutations of the variables, we have
	\begin{equation}		\label{eq_exp_small0}
	\Erw X_q\leq\bink{n}q\cdot\pr\brk{Y(Q_0)\geq2q}
		\leq\exp\brk{q(1+\ln(n/q))}\cdot\pr\brk{Y(Q_0)\geq2q}.
	\end{equation}
Furthermore, the probability that a random $k$-clause contains two variables from $Q_0$
is at most $\bink k2(q/n)^2$
(because for each of the $\bink k2$ pairs of `slots' in the clauses the probability that both
of them are occupied by variables from $Q_0$ is at most $(q/n)^2$).
As $F_k(n,m)$ consists of $m$ independent $k$-clauses, $Y(Q_0)$ is stochastically dominated
by a binomial random variable $\Bin(m,\bink{k}2(q/n)^2)$.
Consequently, assuming that $q/n\leq\chi$ is sufficiently small, we get
	\begin{eqnarray}
	\pr\brk{Y(Q_0)\geq2q}&\leq&
		\pr\brk{\Bin\bc{m,\bink{k}2q}\geq2q}\nonumber\\
		&\leq&\exp\brk{-1.9q\cdot\brk{\ln\bcfr{2q}{\bink{k}2(q/n)^2m}-1}}\qquad\mbox{[by the Chernoff bound~(\ref{eqChernoff})]}
			\nonumber\\
		&\leq&\exp\brk{-1.9q\cdot\ln\bc{\frac{4}{\eul k^2r}\cdot\frac{n}q}}
			\leq\exp\brk{-1.9q\cdot\ln\bc{\frac{4}{\eul k^22^k}\cdot\frac{n}q}}.
		\label{eq_exp_small1}
	\end{eqnarray}
Choosing $\chi=\chi(k)$ sufficiently small, we can ensure that
$(q/n)^{1/4}\leq\chi^{1/4}\leq 4/(\eul k^22^k).$
Plugging this bound into~(\ref{eq_exp_small1}), we get
	\begin{eqnarray}		\label{eq_exp_small2}
	\pr\brk{Y(Q_0)\geq2q}&\leq&
		\exp\brk{-1.1q\cdot\ln\bc{n/q}}.
	\end{eqnarray}
Combining~(\ref{eq_exp_small0}) and~(\ref{eq_exp_small2}), we get
	$\Erw X_q\leq	\exp\brk{-0.1q\ln\bc{n/q}}.$
In effect, $\Erw\sum_{1\leq q\leq\chi n}X_q=O(n^{-0.1})$.
Hence, Markov's inequality implies that \whp\
$\sum_{1\leq q\leq\chi n}X_q=0$,
in which case~(\ref{eq_exp_small}) holds.
\qed\end{proof}

\section{Proof of \Thm~\ref{Thm_vars}}\label{Sec_vars}

\subsection{Loose variables}

Let $\sigma$ be a satisfying assignment of a $k$-CNF $\Phi$.
Remember that a literal $l$ {\em supports} a clause $C$ of $\Phi$ if $l$ is the only literal in $C$ that is true under $\sigma$.
Moreover, we say that a literal $l$ is {\em $1$-loose} if it is true under $\sigma$ and supports no clause.
In addition, $l$ is {\em $2$-loose} if $l$ is true under $\sigma$ and each clause that $l$ supports contains a $1$-loose literal
	from $L_t$.
Thus, any $1$-loose literal is $2$-loose as well.
The key step of the proof is to establish the following.

\begin{proposition}\label{Prop_2loose}
Suppose that $\theta\geq3\exp(\rho)(\ln\rho+10)/k$ and $r\leq 2^k\ln2-k$.
Let $(\Phi,\sigma)$ be a random pair chosen from the distribution $\cU_k\bc{n,m}$.
Then there are at least $0.999\theta n$ $2$-loose literals in $L_t$ \whp
\end{proposition}
To prove \Prop~\ref{Prop_2loose}, we start by estimating the number of $1$-loose variables.

\begin{lemma}\label{Lemma_1loose}
Suppose that $\theta\geq\exp(\rho)/k$ and $\rho\leq k\ln2$.
Let $(\Phi,\sigma)$ be a random pair chosen from the distribution $\cP_k'\bc{n,m}$.
With probability at least $1-\exp(-k2^{2-k}n)$ the number of $1$-loose in $L_t$
	is at least 	$\theta n\cdot\exp(-\rho)/2$.
\end{lemma}
\begin{proof}
By \Lem~\ref{Lemma_SuppBin} the number $X$ of $1$-loose literals in $L_t$ has a binomial distribution with mean
	\begin{eqnarray*}
	\Erw X&=&\theta n\cdot\pr\brk{\Bin\bc{\frac kn\cdot\bink nk,\frac{m}{(2^k-1)\bink nk}}=0}\\
		&=&\theta n\cdot\bc{1-\frac{m}{(2^k-1)\bink nk}}^{\frac{k}n\cdot\bink{n}{k}}
		\sim \theta  n\cdot\exp\bc{-\frac{kr}{2^k-1}}=\theta n\exp(-\rho-\rho/(2^k-1)).
	\end{eqnarray*}
As $\theta\geq\exp(\rho)/k$ and $\rho\leq k\ln2$,
 the Chernoff bound~(\ref{eqChernoff}) shows that for large enough $k$
	\begin{eqnarray*}
	\pr\brk{X<\theta n\exp(-\rho)/2}&\leq&
		\exp\brk{-\frac{\theta n}{8\exp(\rho)}}\leq
			\exp(-k2^{2-k}n),
	\end{eqnarray*}
as desired.
\qed\end{proof}

\begin{lemma}\label{Lemma_2loose}
Suppose that $\theta\geq3\exp(\rho)(\ln\rho+10)/k$,
$\rho\geq\rho_0$ with $\rho_0$ as in \Lem~\ref{Lemma_1loose}, and that $k$ is sufficiently large.
Let $(\Phi,\sigma)$ be a random pair chosen from the distribution $\cP_k'(n,m)$.
Then with probability at least $1-\exp(-k2^{1-k}n)$ the number of $2$-loose literals in $L_t$ is at least $0.999\theta n$.
\end{lemma}
\begin{proof}
To simplify the notation, we are going to condition on 
 $\sigma$ being the all-true assignment; this is without loss of generality.
For each variable $x\in V_t$ we let $S_x$ be the number of clauses supported by $x$.
Moreover, let $S=\sum_{x\in V_t} S_{x}$ and let $X$ be the number of variables $x\in V_t$ such that $S_x=0$.
Thus, $X$ equals the number of $1$-loose variables.

Let $\EE$ be the event that $X\geq \theta n\exp(-\rho)/2$ and $S\leq 2\rho\theta n$.
Since the number of possible clauses with precisely one positive literal in $L_t$
is $\theta n\bink{n-1}{k-1}$, 
$S$ has a binomial distribution $\Bin%
	[\theta n\bink{n-1}{k-1},m/((2^k-1)\bink{n}k)]$.
Therefore, \Lem~\ref{Lemma_1loose} implies that
	\begin{eqnarray}\nonumber
	\pr\brk{\neg\EE}&\leq&\pr\brk{X<\theta n\exp(-\rho)/2}+\pr\brk{S>2\rho\theta n}\\
		&\leq&\exp\brk{-k2^{2-k}n}+\pr\brk{\Bin\bc{\theta n\bink{n-1}{k-1},\frac m{(2^k-1)\bink{n}k}}>2\rho\theta n}.
		\label{eq2loose0}
	\end{eqnarray}
We have
	$$\theta n\bink{n-1}{k-1}\cdot\frac m{(2^k-1)\bink{n}k}\leq \frac{2^k}{2^k-1}\cdot\rho \theta n.$$
Hence, combining~(\ref{eq2loose0}) with the Chernoff bound~(\ref{eqChernoff}), we obtain
for sufficiently large $k$
	\begin{eqnarray}\label{eq2loose1}
	\pr\brk{\neg\EE}&\leq&\exp\brk{-k2^{2-k}n}+\exp\brk{-0.99\rho\theta n}
		\leq2\exp\brk{-k2^{2-k}n},
	\end{eqnarray}
where in the last step we used the assumption that $\rho\geq\rho_0$ for a fixed constant $\rho_0>0$.

Let us now condition on the event that $S=s$ for some number $s\leq2\rho n$,
and on the event $\cE$.
In this conditional distribution
for each of the $s$ clauses supported by some variable in $V_t$ the $k-1$ negative literals that the clause contains are independently uniformly distributed.
Therefore, for each such clause the number of negative literals $\bar y$ whose underlying variable $y$ is $1$-loose
is binomially distributed $\Bin(k-1,X/n)$.
Consequently, the number $T$ of clauses supported by some variable in $V_t$ in which no $1$-loose variable occurs negatively
has a binomial distribution with mean
	$s\cdot\pr\brk{\Bin(k-1,X/n)=0}$.
Hence,
	\begin{eqnarray*}
	\Erw\brk{T|\EE}&\leq&
		2\rho\theta n\cdot\pr\brk{\Bin(k-1,\theta\exp(-\rho)/2)=0}\\
		&=&2\rho\theta n \cdot(1-\theta\exp(-\rho)/2)^{k-1}
			\leq2\rho\theta n \exp(-\theta\exp(-\rho)k/3)			\leq2\exp(-10)\theta n.
	\end{eqnarray*}
Thus, the Chernoff bound~(\ref{eqChernoff}) implies that for $k\geq k_0$ large enough
	\begin{eqnarray}
	\pr\brk{T>0.001\theta n|\EE}&\leq&
		\exp(-0.001\theta n)\leq\exp\brk{-k2^{2-k}n}.
		\label{eq2loose2}
	\end{eqnarray}
Finally, the assertion follows from~(\ref{eq2loose1}) and~(\ref{eq2loose2}).
\qed\end{proof}

\begin{proof}[\Prop~\ref{Prop_2loose}]
Let $\cE$ be the event that a pair $(F,\sigma)\in\Lambda_{n,m}$ has at least $0.999\theta n$ $2$-loose literals.
\Lem~\ref{Lemma_2loose} shows that 
	\begin{equation}\label{eqProp2looseFinal}
	\pr_{\cP_k'\bc{n,m}}\brk\cE\geq1-\exp(-k2^{1-k}n)\geq1-\exp(-krn/4^k).
	\end{equation}
Moreover, \Cor~\ref{Cor_Pnm'} and~(\ref{eqProp2looseFinal}) imply that
	$\pr_{\cU_k\bc{n,m}}\brk\cE=1-o(1)$
as desired.
\qed\end{proof}

\begin{proof}[\Thm~\ref{Thm_vars}, part 1]
By Fact~\ref{XFact_UD} it suffices to prove the desired statement for the experiment {\bf U1}--{\bf U4}.
Thus, let $(\Phi,\sigma)$ be a pair chosen from the distribution $\cU_k\bc{n,m}$.
Without loss of generality we may condition on $\sigma$ being the all-true assignment.
Let $\cL$ be the set of all tame variables that are $2$-loose.
Then by \Prop s~\ref{Prop_tame} and~\ref{Prop_2loose} we have $\cL\geq(0.999-o(1))\theta n\geq0.99\theta n$ \whp\
Assuming that this is the case,
we are going to show that if $x\in \cL$, then
there is a satisfying assignment $\tau$ such that $\tau(x)\not=\sigma(x)$ and $\dist(\tau,\sigma)\leq\ln(n)$.

Thus, fix a variable $x\in\cL$.
If $x$ is $1$-loose, then we can just set $\tau(x)=1-\sigma(x)=0$ and $\tau(y)=\sigma(y)=1$ for all $y\neq x$
to obtain a satisfying assignment with $\dist(\tau,\sigma)=1$,
because $x$ does not support any clauses.
Hence, assume that $x$ is $2$-loose but not $1$-loose.
Let $\cC$ be the set of all clauses supported by $x$ in $(\Phi,\sigma)$.
Any clause $C\in\cC$ contains a negative occurrence of a $1$-loose variable $x_C\in V_t$ in $C$
	(by the very definition of $2$-loose).
Define
	$\tau(x)=0$, $\tau(x_C)=0$ for all $C\in\cC$,
	and $\tau(y)=\sigma(y)=1$ for all other variables $y$.

We claim that $\tau$ is a satisfying assignment.
To see this, assume for contradiction that there is a clause $U$ that is unsatisfied under $\tau$.
Then $U$ contains a variable from $\cbc x\cup\cbc{x_C:C\in\cC}$ positively, while none of these variables
occurs negatively in $U$.
Hence, $U\not\in\cC$.
Moreover, since the variables $x_C$, $C\in\cC$, do not support any clauses,
$U$ indeed contains two variables from the set $\cbc x\cup\cbc{x_C:C\in\cC}$ positively.
There are two possible cases.
\begin{description}
\item[Case 1: $x$ occurs in $U$.]
	Let $C\in\cC$ such that $x_C$ occurs in $U$ as well.
	Then the factor graph contains the cycle $x,C,x_C,U,x$, in contradiction to our assumption that $x$ is tame.
\item[Case 2: $x$ does not occur in $U$.]
	There exist $C_1,C_2\in\cC$ such that $x_{C_1},x_{C_2}$ occur in $C$.
	Hence, the factor graph contains the cycle
		$x,C_1,x_{C_1},C,x_{C_2},C_2,x$,
	once more in contradiction to the assumption that $x$ is tame.
\end{description}
Hence, there is no clause $U$ that is unsatisfied under $\tau$.
Finally, since all the variable $x_C$ with $C\in\cC$ have distance two from $x$ in the factor graph,
and as $x$ is tame, we have $\dist(\sigma,\tau)\leq\ln n$.
\qed\end{proof}

\subsection{Rigid variables}\label{Sec_rigid}

The proof of the second part of \Thm~\ref{Thm_vars} follows the outline given in \Sec~\ref{Sec_rigidOutline}.
Recall the function $\varphi$ from~(\ref{eqvarphi}).

\begin{proposition}\label{Prop_Supp}
Suppose that $k\geq 6$ and $0<r\leq 2^k\ln2-k$.
Let $\mu=\rho\cdot 2^k/(2^k-1)$ and $\zeta=(1+\mu+\mu^2/2)/\exp(\mu)$, and assume that
	$2^k\theta\zeta\varphi(1)>\rho$.
Then \whp\ in a random pair $(\Phi,\sigma)$ chosen from the distribution $\cU_k\bc{n,m}$ no more than
$2\zeta\theta n$ literals in $L_t$ support fewer than three clauses.
\end{proposition}
\begin{proof}
Let $S$ be the number of literals $l\in L_t$ that support fewer than three clauses.
We are going to show that
	\begin{equation}\label{eqPropSupp1}
	\pr_{\cP_k\bc{n,m}}\brk{S>2\zeta\theta n}\leq\exp(-krn/4^k).
	\end{equation}
Then \Cor~\ref{Cor_Pnm'} implies the assertion.

In the distribution $\cP_k'\bc{n,m}$ the random variable $S$ is binomially distributed
with mean $(1+o(1))\theta\zeta n$ by the second part of \Lem~\ref{Lemma_SuppBin}.
Hence, the Chernoff bound~(\ref{eqChernoff}) shows that
	\begin{eqnarray}\label{eqexSupp1}
	\pr_{\cP_k\bc{n,m}}\brk{S>2\zeta n}&\leq&
		\exp\bc{-(1+o(1))\theta\zeta\varphi(1)n}.
	\end{eqnarray}
By the assumptions on $\mu$ and $\theta$ we have 
	$\theta\zeta\varphi(1)>\rho/2^k$;
hence, (\ref{eqPropSupp1}) follows from~(\ref{eqexSupp1}).
\qed\end{proof}

Remember that a set $\cS\subset L_t$ of literals $t$-{\em self-contained}
if each literal $l\in\cS$ supports at least two clauses that contain
literals from $\cbc{x_1,\bar x_1,\ldots,x_t,\bar x_t}\cup \cS\cup\bar\cS$ only, where $\bar\cS$ is the set of all negations of literals in $\cS$.

\begin{proposition}\label{Prop_frozen}
For any $k\geq3$ there is a number $\chi=\chi(k)>0$ such that for any
$0<r\leq2^k\ln2-k$ the following is true.
Let $(\Phi,\sigma)$ be a random pair chosen from the distribution $\cU_k\bc{n,m}$.
Then \whp\ for any $t$-self-contained set $\cS$ all variables $x\in\cS\cup\bar\cS$ are $\chi n$-rigid.
\end{proposition}
\begin{proof}
Let $(\Phi,\sigma)$ be a random pair chosen from the distribution $\cU_k\bc{n,m}$.
Without loss of generality we may condition on $\sigma$ being the all-true assignment.
By \Lem~\ref{Lemma_exp_small} there is a number
$\chi=\chi(k)>0$ such that~(\ref{eq_exp_small}) is satisfied \whp,
and we are going to assume that this is the case.

Let $\cS$ be a self-contained set.
Suppose that $\tau$ is a satisfying assignment such that the set
$Q$ of all variables $x\in\cS\cup\bar\cS$ such that $\tau(x)\not=\sigma(x)$ is non-empty.
For each variable $x\in Q$ there are two clauses $C_1(x),C_2(x)$ that
are supported by $x$ in $\sigma$ and that consist of literals from $\cS\cup\bar\cS$ only (because $\cS$ is self-contained).
Since $\tau$ is satisfying and $\tau(x)\not=\sigma(x)$, both $C_1(x)$ and $C_2(x)$ contain another variable from $Q$.
Hence, there are at least $2|Q|$ clauses that contain at least two variables from $Q$.
Thus, (\ref{eq_exp_small}) implies that $|Q|>\chi n$, and consequently
$\dist(\sigma,\tau)\geq|Q|>\chi n$.
\qed\end{proof}

\begin{proposition}\label{Prop_RandPoisson}
Suppose that $k\geq4$ and $0<r\leq2^k\ln2-k$, and that $0\leq\theta\leq1$.
Set
	$$\mu=\frac{\rho 2^k}{2^k-1},\ \zeta=\frac{1+\mu+\mu^2/2}{\exp(\mu)},\ 
		\lambda=1-(1-3\theta\zeta)^{k-1},\
		\gamma=\frac{\mu\cdot\bc{\exp(\lambda\mu)-1-\lambda\mu}}{(1-\zeta)\exp(\mu)}$$
and let $h(x)=-x\ln x-(1-x)\ln(1-x)$. If $\zeta<1/3$ and
	\begin{equation}\label{eqRandPoisson}
	\theta (\zeta\ln(\gamma)+h(\zeta))+\rho/2^k<0,
	\end{equation}
then a random pair $(\Phi,\sigma)$ chosen from the distribution $\cU_k\bc{n,m}$ has one of the following properties \whp
\begin{enumerate}
\item[a.] More than $2\theta\zeta n$ literals in $L_t$
		 that are true under $\sigma$ literals support fewer than three clauses.
\item[b.] There is a $t$-self-contained set of size $(1-3\zeta)\theta n$.
\end{enumerate}
\end{proposition}
\begin{proof}
Let $p=m/((2^k-1)\bink nk)$.
Let $(\Phi,\sigma)$ be chosen from the distribution $\cP_k\bc{n,m}$.
We may condition on $\sigma$ being the all-true assignment, and on the event that
at most $2\zeta\theta n$ literals 
amongst $x_{t+1},\bar x_{t+1},\ldots,x_n,\bar x_n$
that are true under $\sigma$ support fewer than three clauses
(as otherwise a.\ occurs).
In fact, fix a set $Z$ of $2\zeta\theta  n$ variables and
condition on the event $\cE$ that all variables that support at most two clauses lie in $Z$.
For any variable $x\not\in Z$ we let $S_x$ be the number of clauses supported by $x$.
Then the first part of~\Lem~\ref{Lemma_SuppBin} implies that $S_x$ has
a binomial distribution $\Bin(k\bink{n-1}{k-1},p)$ conditioned on the outcome being at least three.
As a consequence, 
for any $j\geq3$
	\begin{eqnarray}
	\pr\brk{S_x=j|\cE}&=&\frac{\pr\brk{\Bin(k\bink{n-1}{k-1},p)=j}}{\pr\brk{\Bin(k\bink{n-1}{k-1},p)<3}}
		\leq\frac{(1+o(1))\mu^j}{j!\exp(\mu)(1-\zeta)}+O\bc{\exp(-j/\mu)/n}.
		\label{eqRandPoisson1}
	\end{eqnarray}

Let $X\subset \cbc{x_{t+1},\ldots,x_n}\setminus Z$ be a set of $\zeta\theta  n$ variables.
For each $x\in X$ we let $T_x$ be the number of clauses supported by $x$ in which
a variable from $X\cup Z$ occurs negatively.
In a random clause supported by $x$ the variables underlying the $k-1$ negative literals in that clause are distributed
uniformly over $V$.
Therefore, given $\cE$ the probability that such a clause contains at least one variable from $X\cup Z$ is
	$$1-(1-|X\cup Z|/n)^{k-1}+o(1)=1-(1-3\theta\zeta)^{k-1}+o(1)\sim\lambda.$$
Hence, if we condition on both $\cE$ and $S_x=j$, then the probability that $T_x\geq S_x-1$ equals
$j\cdot(\lambda+o(1))^{j-1}$.
Thus, letting $\gamma=\mu\cdot\bc{\exp(\lambda\mu)-1-\lambda\mu}/(1-\zeta)\exp(\mu)$,
we obtain from (\ref{eqRandPoisson1})
	\begin{eqnarray*}
	\pr\brk{T_x\geq S_x-1|\cE}&=&\sum_{j\geq3}\pr\brk{T_x\geq S_x-1|\cE\mbox{ and }S_x=j}\cdot\pr\brk{S_x=j|\cE}\\
		&\leq&(1+o(1))\sum_{j\geq3}\frac{j\lambda^{j-1}\mu^j}{j!\exp(\mu)(1-\zeta)}
			\sim\gamma.
	\end{eqnarray*}
Given that $\cE$ occurs the events $T_x\geq S_x-1$ are mutually independent for all $x\in X$.
Therefore,
	\begin{eqnarray*}
	\pr\brk{\forall x\in X:T_x\geq S_x-1|\cE}&\leq&(\gamma+o(1))^{\theta\zeta n}.
	\end{eqnarray*}
If b.\ does not occur, then there is a set $X\subset V\setminus Z$ of size $\zeta n$ such that $T_x\geq S_x-1$
for all $x\in X$.
Hence, by the union bound the probability that b.\ does not occur is at most
	\begin{eqnarray*}
	\pr\brk{\exists X\subset V\setminus Z,\,|X|=\zeta n:\forall x\in X:T_x\geq S_x-1|\cE}
		&\leq&\hspace{-3mm}\sum_{X\subset V\setminus Z,\,|X|=\theta\zeta n}\hspace{-3mm}
			\pr\brk{\forall x\in X:T_x\geq S_x-1|\cE}\\
	&\hspace{-12cm}\leq&\hspace{-6cm}\;
		\bink{(1-2\zeta)\theta n}{\theta\zeta n}\cdot(\gamma+o(1))^{\theta\zeta n}
		\leq\exp\brk{(1-2\zeta)h(\zeta/(1-2\zeta))\cdot n}\cdot(\gamma+o(1))^{\zeta n}\\
	&\hspace{-12cm}\leq&\hspace{-6cm}\;
		\exp\brk{\theta n\cdot\bc{(1-2\zeta)h(\zeta/(1-2\zeta))+\zeta\cdot\ln\gamma+o(1)}}<\exp(-\rho n/2^k)\qquad
			\mbox{[by~(\ref{eqRandPoisson})]},
	\end{eqnarray*}
as desired.
Finally, the assertion follows directly from 
\Cor~\ref{Cor_Pnm'}.
\qed\end{proof}

\begin{proof}[\Thm~\ref{Thm_vars}, part~2]
Suppose that $1/k\leq\theta\leq\exp(\rho)/(\rho^3k)$.
The goal is to verify~(\ref{eqRandPoisson}).
Since
	$h(\zeta)\leq\zeta(1-\ln\zeta),$
proving (\ref{eqRandPoisson}) reduces to showing
	$\theta\zeta\brk{\ln\gamma+1-\ln\zeta}<-\rho/2^k,$
i.e.,
	\begin{equation}\label{eqRandPoissonGen}
	\theta\zeta\ln(\eul\gamma/\zeta)<-\rho/2^k.
	\end{equation}
Plugging in the definitions of $\gamma$ and $\zeta$, we see that
	\begin{eqnarray*}
	\ln\bcfr{\eul\gamma}\zeta&=&\ln\brk{\frac{\eul\mu(\exp(\lambda\mu)-\lambda\mu-1)}{(1-\zeta)(1+\mu+\mu^2/2)}}
		\leq\ln\bcfr{5\eul\brk{\exp(\lambda\mu)-\lambda\mu-1}}{\mu}\qquad\mbox{[for $\mu$ not too small].}
	\end{eqnarray*}
Since $3\theta\zeta\leq4/(k\rho)$ for $\rho\geq\rho_0$ sufficiently large,
we have $\lambda=1-(1-3\theta\zeta)^{k-1}\leq4k\theta\zeta$.
Hence,
	$$\lambda\mu\leq 4k\mu\theta\zeta\leq\frac{4\mu^2}{\exp(\mu)}\cdot\frac{\exp(\mu)}{\mu^3}\leq4/\mu.$$
Therefore, we obtain for $\mu\geq\rho\geq\rho_0$ large
	\begin{eqnarray*}
	\ln\bcfr{\eul\gamma}\zeta&\leq&\ln\bcfr{4\eul(\lambda\mu)^2}{\mu}\leq\ln\bc{64\eul/\mu^3}\leq
		-1.
	\end{eqnarray*}
As $\theta\zeta\geq\zeta/k\geq\frac12\mu^2\exp(-\mu)\geq\frac13\rho^2/2^{k}$
for $k\geq k_0$ and $\rho\geq\rho_0$ not too small,
	we thus obtain~(\ref{eqRandPoissonGen}).
\qed\end{proof}

\subsection{Forced variables}\label{Sec_unit}

Let $(\Phi,\sigma)$ be a formula/assignment pair.
A clause $C$ \emph{forces} a variable $x\in V_t$ if $C$ contains $k-1$ literals
from $\cbc{x_1,\bar x_1,\ldots,x_t,\bar x_t}$, none of which satisfies $C$ under $\sigma$,
and either the literal $x$ or $\bar x$, which does.

\begin{lemma}\label{Lemma_forced}
Suppose that $\rho\geq\rho_0$, $k\geq k_0$, and $k\theta\sim\ln(\rho)-10$.
Then \whp\ in a pair $(\Phi,\sigma)$ chosen from the distribution $\cU_k(n,m)$
at least $0.991\theta n$ variables in $V_t$ are forced.
\end{lemma}
\begin{proof}
Let $\cF$ be the event that at least $0.991\theta n$ variables in $V_t$ are forced.
We are going to show that
	\begin{equation}\label{eqforced1}
	\pr_{\cP_k'(n,m)}\brk{\cF}\geq1-\exp(-1.1\rho/2^k),
	\end{equation}
so that the assertion follows from \Cor~\ref{Cor_Pnm'}.

Thus, let $(\Phi',\sigma')$ be a pair chosen from the distribution $\cP_k'(n,m)$.
We may assume without loss of generality that $\sigma'$ is the all-true assignment.
For each variable $x\in V_t$ the number of clauses that $x$ supports has a
binomial distribution with mean $\mu=\rho\cdot 2^k/(2^k-1)$.
Furthermore, if $C$ is a random clauses supported by $x$,
then $C$ contains $k-1$ random negative literals;
the probability that all of these are in $V\setminus V_t$ equals $(1-\theta+o(1))^{k-1}$.
Hence, the number $F_x$ of forcing clauses for $x$ is binomially distributed with mean
	\begin{eqnarray*}
	\Erw\brk{F_x}&=&\mu(1-\theta+o(1))^{k-1}\geq\rho(1-\theta)^{k-1}\\
		&\geq&\rho\exp\brk{-(\theta+\theta^2)(k-1)}\geq\rho\exp\brk{-\theta k-\theta^2k}\geq\exp(5).
	\end{eqnarray*}
Therefore, for any $x\in V_t$ we have
	$\pr\brk{F_x=0}\leq\exp(-\exp(5)),$
and the events $(\cbc{F_x=0})_{x\in V_t}$ are mutually independent.
Hence, the number $Z$ of variables $x\in V_t$ with $F_x=0$ is binomially distributed
with mean $\exp(-\exp(-5))\theta n$, and thus
	$$\pr\brk{Z\geq0.009\theta n}\leq\exp(-0.009\theta n)\leq\exp(-1.1\rho/2^k)$$
by Chernoff bounds.
This proves~(\ref{eqforced1}).
\qed\end{proof}

\begin{proof}[\Thm~\ref{Thm_vars}, part~3]
To complete the proof of \Thm~\ref{Thm_vars}, part~3, we need to deal
with general values $1/n\ll\theta\leq\theta_0=(\ln(\rho)-10)/k$.
Let $t=(1-\theta)n$ and $t_0=(1-\theta_0)n$.
To obtain a pair $(\Phi_t,\sigma_t)$ from the distribution {\bf U1}--{\bf U4},
one can proceed as follows.
First, choose a pair $(\Phi_{t_0},\sigma_{t_0})$ from the distribution {\bf U1}--{\bf U4} with $t_0$ variables decimated.
Then, assign the variables in $x\in V_{t_0}\setminus V_t$ with the truth values
	$\sigma_{t_0}(x)$, simplify the formula,
and let $\sigma_t(y)=\sigma_{t_0}(y)$ for all $y\in V_t$.
We are going to use this experiment to analyze the number of forced variables in $(\Phi_t,\sigma_t)$.

The above experiment shows that any variable $x\in V_t$ that is forced in $(\Phi_{t_0},\sigma_0)$ remains forced in $(\Phi_t,\sigma_t)$.
Let $\cF$ be the set of forced variables in $(\Phi_{t_0},\sigma_0)$.
Given that $|\cF|=j$, the set $\cF$ is a uniformly random subset of $V_{t_0}$.
Hence, if we condition on the event that $\abs{\cF}\geq0.991\theta_0 n$,
then $\abs{\cF\cap V_t}$ has a hypergeometric distribution with mean at least $0.991\theta n$.
Therefore, by Chebyshev's inequality, we have $\abs{\cF\cap V_t}\geq(0.991\theta-o(1)) n\geq0.99\theta n$ \whp\
(here we use that $\theta n\gg1$).
Thus, the theorem follows from \Lem~\ref{Lemma_forced}.
\qed\end{proof}

\section{Proof of \Thm~\ref{Thm_shape}}\label{Sec_shattering}

\subsection{Shattering}

In this section we prove the first part of \Thm~\ref{Thm_shape}.
Consider a pair $(\Phi,\sigma)$ chosen from the planted model $\cP_k(n,m)$.
Let $\Phi_{t,\sigma}$ denote the formula obtained from $\Phi$ by substituting
the values $\sigma(x_1),\ldots,\sigma(x_t)$ for the first $t$ variables.
Without loss of generality, we may assume that $\sigma=\vecone$ is the all-true assignment.
The main step of the proof is the summarized in the following proposition.

\begin{proposition}\label{Prop_icy_sat}
Let $k\geq6$ and $r>0$ be fixed.
Moreover, let  $0<\theta\leq1$ and let
	$$\psi:(0,1)\rightarrow\RR,\quad\alpha\mapsto-\alpha\theta\ln\alpha-(1-\alpha)\theta\ln(1-\alpha)+
			r\ln\bc{1-\frac{1-(1-\alpha\theta)^k}{2^k-1}}.$$
Suppose that there is a number $a\in(0,1)$  such that
	\begin{equation}\label{eqpsiShatters}
	\psi(a)+\rho/2^k<0\quad\mbox{ and }
		\sup_{0<\alpha<a}\psi(\alpha)<\theta\ln2+2^k\rho\ln(1-2^{-k})/k-\rho/2^k.
	\end{equation}
Then there is $\eps=\eps(k,\rho)$ such that for $\Phi_t$ generated by the experiment {\bf U1}--{\bf U4},
the set $\cS(\Phi_t)$ is $(a-\eps,a+\eps)$-shattered.
\end{proposition}

We will first show how \Prop~\ref{Prop_icy_sat} implies the first part of \Thm~\ref{Thm_shape}.
The proof of \Prop~\ref{Prop_icy_sat} appears at the end of this section.
To derive the first part of \Thm~\ref{Thm_shape} from \Prop~\ref{Prop_icy_sat}
 we need to verify~(\ref{eqpsiShatters}). 

\begin{lemma}\label{Lemma_a}
Assume that
	$0\leq\theta\leq\exp(\rho-2)/(\rho k)$.
Let $a=\exp(2-\rho)$.
Then $\psi(a)<-a\theta/2$.
\end{lemma}
\begin{proof}
We have
	\begin{eqnarray*}
	\psi(a)&\leq&a\theta(1-\ln a)-\frac{\rho}k\bc{1-(1-a\theta)^k}
			\leq a\theta(1-\ln a)-\frac{\rho}k\bc{1-\exp(-ak\theta)}\\
		&\leq&a\theta(1-\ln a)-\frac{\rho}k\bc{ak\theta-(ak\theta)^2/2}=
			a\theta\brk{1-\ln a-\rho(1-a k\theta/2)},
	\end{eqnarray*}
where we used $\exp(-z)\leq1-z+z^2/2$ for $z\geq0$.
Since $k\theta\rho\leq\exp(\rho-2)$ by assumption,
our choice of $a$ implies that
	$
	\psi(a)\leq
		a\theta\brk{1-\ln a-\rho+a\exp(\rho-2)/2}
			=-a\theta/2,
	$
as claimed.
\qed\end{proof}

\begin{lemma}\label{Lemma_alpha}
Assume that
	$0\leq\theta\leq\exp(\rho-2)/(\rho k)$.
Let $a=\exp(2-\rho)$.
Then $\sup_{\alpha<a}\psi(\alpha)\leq \frac{3}{2\eul^2k\rho}$.
\end{lemma}
\begin{proof}
Let $0\leq\alpha< a$.
We have
	\begin{eqnarray*}
	\psi(\alpha)&\leq&
		\theta(\alpha-\alpha\ln\alpha-\alpha\rho(1-\alpha k\theta/2)).
	\end{eqnarray*}
Let $\psi_1(\alpha)$ be the expression on the r.h.s.
Then
	\begin{eqnarray*}
	\frac d{d\alpha}\psi_1(\alpha)=\theta\brk{-\ln\alpha-\rho+\alpha k\theta},&&
	\frac{d^2}{d\alpha^2}\psi_1(\alpha)=\theta\brk{k\theta-1/\alpha}.
	\end{eqnarray*}
Thus, our assumption on $\theta$ implies that $\frac{d^2}{d\alpha^2}\psi_1(\alpha)<0$ for all $0<\alpha<a$, and therefore
$\psi_1$ has a unique local maximum in the interval $(0,\alpha)$.
To pinpoint this maximum, note that for $\alpha_0=\exp(-\rho)$ the first derivative $\frac d{d\alpha}\psi_1(\alpha_0)$ is positive.
Moreover, at $\alpha_1=\exp(1-\rho)$ we have $\frac d{d\alpha}\psi_1(\alpha_1)<0$.
Hence, the unique local maximum of $\psi_1$ lies in the interval $(\alpha_0,\alpha_1)$.
To study the maximum value, consider the function
	$\psi_2:\alpha\mapsto\alpha-\alpha\ln\alpha-\alpha\rho$.
Its derivative is $d/d\alpha\,\psi_2(\alpha)=\rho-\ln\alpha$, so that the maximum of this function occurs at $\alpha_0$.
Furthermore, the quadratic term $\alpha\mapsto\alpha^2k/2$ is monotonically increasing in $\alpha$.
Therefore,
	\begin{eqnarray*}
	\sup_{0<\alpha<a}\psi(\alpha)&\leq&\sup_{0<\alpha<a}\psi_1(\alpha)
		=\sup_{\alpha_0<\alpha<\alpha_1}\psi_1(\alpha)
		\leq\theta(\psi_2(\alpha_0)+\alpha_1^2k/2)
		=3\theta\exp(-\rho)/2.
	\end{eqnarray*}
Finally, the assertion follows from the assumed bound on $\theta$.
\qed\end{proof}

\begin{proof}[\Thm~\ref{Thm_shape}, part~2]
Assume that 
$\rho\leq k\ln 2-\ln k$ and
	$$\frac{\rho}{k\ln2}(1+\rho^{-2}+2^{2-k})\leq\theta\leq\exp(\rho-2)/(\rho k).$$
Let $a=\exp(2-\rho)$.
\Lem~\ref{Lemma_a} shows that
	\begin{eqnarray*}
	\psi(a)+\rho/2^k&\leq&\rho/2^k-\exp(2-\rho)\theta/2
		\leq \rho/2^k-\frac{\exp(2-\rho)\rho}{k\ln2}
			=\frac{\rho}{2^k}\bc{1-\frac{2^k\exp(2-\rho)}{k\ln2}}.
	\end{eqnarray*}
Since $\rho\leq k\ln2-\ln k$, the r.h.s.\ is negative.
By \Lem~\ref{Lemma_alpha} we have 
	\begin{eqnarray}\nonumber
	\theta\ln2+\frac{2^k\rho}{k}\ln(1-2^{-k})-\rho/2^k
		&\geq&\theta\ln2-\frac{\rho}{k}-\rho/2^{k-1}\\
		&\geq&\frac{1}{k\rho}+2^{2-k}\rho\ln2-\rho/2^{k-1}\geq\frac{1}{k\rho}>\sup_{\alpha<a}\psi(\alpha).
			\label{eqClusterEntropy}
	\end{eqnarray}
Thus, the assertion follows from \Prop~\ref{Prop_icy_sat}.
\qed\end{proof}

\subsubsection{Proof of \Prop~\ref{Prop_icy_sat}.}\label{Sec_icy_sat}

In the rest of this section we keep the notation and the assumptions from \Prop~\ref{Prop_icy_sat}.
Let 
	$$b=\theta\ln2+2^k\rho\ln(1-2^{-k})/k-\rho/2^k.$$

\begin{lemma}\label{Lemma_icy_sat}
There exist numbers $\xi>0$, $0<a_1<a_2<1$ such that 
a pair $(\Phi,\sigma)$ chosen from the distribution $\mathcal{P}_k\bc{n,m}$ has the following two
properties with probability at least $1-\exp(-(\xi+\rho/2^k)n)$.
\begin{enumerate}
\item $\Phi_{t,\sigma}$ does not have a satisfying assignment $\tau$ with $a_1n<\dist(\sigma,\tau)<a_2n$.
\item $|\{\tau\in \cS(\Phi_{t,\sigma}):\dist(\sigma,\tau)<a_2n\}|\leq\exp((b-\xi)n)$.
\end{enumerate}
\end{lemma}
\begin{proof}
For $\alpha>0$ we let
	$X_\alpha=\abs{\cbc{\tau\in\cS_{t,\sigma}(F):\dist(\sigma,\tau)=\alpha\theta n}}.$
Note that
	\begin{eqnarray}\label{eqicy1}
	\Erw X_\alpha&\leq&\bink{\theta n}{\alpha\theta  n}\bc{1-\frac{1-(1-\alpha\theta)^k}{2^k-1}}^m,
		\label{eqicy2}
	\end{eqnarray}
Taking logarithms and bounding the binomial coefficient via Stirling's formula, we obtain
	\begin{eqnarray}\label{eqicy3}
	\frac{\ln\Erw X_\alpha}n&\leq&\psi(\alpha).
		\label{eqicy4}
	\end{eqnarray}

Let $a\in(0,1)$ be such that $\psi(a)+\rho/2^k<0$ (cf.~(\ref{eqpsiShatters})).
As $\psi$ is continuous there exist $0<a_1<a<a_2<1$ and $\xi_1>0$ such that
	\begin{equation}\label{eqicy5}
	\sup_{a_1\leq\alpha\leq a_2}\psi(\alpha)<-\rho/2^k-2\xi_1.
	\end{equation}
Combining~(\ref{eqicy4}) and~(\ref{eqicy5}), we conclude that
	$\Erw X_{\alpha}\leq
		\exp\brk{-n(\rho/2^k+2\xi_0)}$
for all $a_1\leq\alpha\leq a_2$.
Summing over integers $a_1n\leq j\leq a_2n$, we see that for large $n$
	\begin{eqnarray*}
	\sum_{a_1n\leq j\leq a_2n}\Erw X_{j/n,\xi_1}&\leq&
		n\exp\brk{-n(\rho/2^k+2\xi_0)}\leq\exp\brk{-n(\rho/2^k+\xi_0)}.
	\end{eqnarray*}
Hence, by Markov's inequality the probability that there is a satisfying assignment $\tau$ that coincides with $\sigma$
on the first $t$ variables such that
$a_1n\leq\dist(\sigma,\tau)\leq a_2n$ is bounded by $\exp(-n(\rho/2^k+\xi_0))$.
This proves the first assertion.

Since we are assuming that $\sup_{0<\alpha<a}\psi(\alpha)<b-\rho/2^k$, 
and as~(\ref{eqicy5}) shows that
	$\psi(\alpha)<-\rho/2^k-2\xi_1<b-\rho/2^k-2\xi_1$ for all $a\leq \alpha<a_2$,
there is a number $\xi_2>0$ such that
	$$\sup_{0<\alpha\leq a_2}\psi(\alpha)<b-\rho/2^k-3\xi_2.$$
Hence, (\ref{eqicy3}) implies that
	$$\Erw X_\alpha\leq\exp(n\psi(\alpha))\leq
		\exp(n(b-\rho/2^k-3\xi_2))\qquad\mbox{for all $0<\alpha\leq a_2$}.$$
Taking the sum over integers $0\leq j\leq a_2n$, we get for large enough $n$
	$$\sum_{0\leq j\leq a_2n}\Erw X_{j/n}\leq
			n\exp(n(b-\rho/2^k-3\xi_2))\leq\exp(n(b-\rho/2^k-2\xi_2)).$$
That is, the expected number of assignments $\tau\in\cS(\Phi_{t,\sigma})$ such that $\dist(\sigma,\tau)\leq a_2n$
is bounded by $\exp(n(b-\rho/2^k-2\xi_2))$.
Hence, Markov's inequality entails that with probability at least $1-\exp(-n(\rho/2^k+\xi_2))$ there
are at most $\exp(n(b-\xi_2))$ such satisfying assignments $\tau$.
This proves the second assertion.
\qed\end{proof}

\begin{corollary}\label{Cor_icy_sat}
There exist numbers $\xi>0$, $0<a_1<a_2<1$ such that 
a pair $(\Phi,\sigma)$ chosen from the distribution $\mathcal{U}_k\bc{n,m}$ enjoys the two
properties stated in Lemma~\ref{Lemma_icy_sat} with probability at least $1-\exp(-\xi n)$.
\end{corollary}
\begin{proof}
This follows directly from Lemma~\ref{Lemma_icy_sat} and \Cor~\ref{Cor_SATExchange}.
\qed\end{proof}

\begin{proof}[Proposition~\ref{Prop_icy_sat}]
Let $\xi,a_1,a_2$ be the numbers provided by Corollary~\ref{Cor_icy_sat} and
let $(\Phi,\sigma)$ be a pair chosen from the distribution $\cU_k(n,m)$.
With each assignment $\tau\in \cS(\Phi_{t,\sigma})$ we associate a set
	$$\mathcal{C}(\tau)=\{\chi\in \cS(\Phi_{t,\sigma}):\dist(\chi,\tau)\leq a_1 n\}.$$
Moreover, we call $\tau\in\cS(\Phi_{t,\sigma})$ \emph{good} if
	$\abs{\mathcal{C}(\tau)}\leq\exp((b-\xi)n)$
and there is no $\chi\in\cS(\Phi_{t,\sigma})$ such that $a_1n\leq\dist(\chi,\tau)\leq a_2n$. 
Let $\cS_{good}$ be the set of all good $\tau\in \cS(\Phi_{t,\sigma})$ and $\cS_{bad}=\cS(\Phi_{t,\sigma})\setminus\cS_{good}$.
Corollary~\ref{Cor_icy_sat} and our choice of $b$ ensure that $F$ has the following two properties \whp:
	\begin{eqnarray}\label{eqIcySumm1}
	\abs{\cS(\Phi_{t,\sigma})}&\geq&2^{t}\exp(bn),\\
	\abs{\cS_{good}}&\geq&(1-\exp(-\xi n))\cdot\abs{\cS(\Phi_{t,\sigma})}.
		\label{eqIcySumm2}
	\end{eqnarray}

Assuming that~(\ref{eqIcySumm1}) and~(\ref{eqIcySumm2})
hold and that $n$ is sufficiently large, we are going to construct a decomposition of
$\cS(\Phi_{t,\sigma})$ into subsets as required by {\bf SH1}--{\bf SH2}.
To this end, choose some $\sigma_1\in\cS_{good}$.
Having defined $\sigma_1,\ldots,\sigma_l$, we choose an arbitrary $\sigma_{l+1}\in\cS_{good}\setminus\bigcup_{j=1}^l\cC(\sigma_j)$,
unless this set is empty, in which case we stop.
Let $\sigma_1,\ldots,\sigma_N$ be the resulting sequence and define
	$$R_l=\cC(\sigma_l)\setminus\bigcup_{j=1}^{l-1}\cC(\sigma_j)\qquad\mbox{for }1\leq l\leq N,
		\quad\mbox{and }R_0=\cS(\Phi_{t,\sigma})\setminus\bigcup_{l=1}^NR_l.$$
Then $\cS(\Phi_{t,\sigma})= R_0\cup\cdots\cup R_N$.
(Observe that possibly $R_0=\emptyset$ while $R_l\not=\emptyset$ for all $1\leq l\leq N$ as $\sigma_l\in R_l$.)
Furthermore, for each $1\leq l\leq N$ we have
	$R_l\subset\cC(\sigma_l)$ and thus
	\begin{eqnarray}\nonumber
	|R_l|&\leq&\abs{\cC(\sigma_l)}\leq\exp((b-\xi)n)\qquad\mbox{[because $\sigma_l$ is good]}\\
		&\leq&\abs{\cS(\Phi_{t,\sigma})}\cdot\exp(-\xi n)\qquad\qquad\qquad\mbox{[by~(\ref{eqIcySumm1})].}
		\label{eqIcySumm3}
	\end{eqnarray}
Furthermore, as $R_0\subset\cS_{bad}$, (\ref{eqIcySumm2}) implies
	\begin{equation}		\label{eqIcySumm4}
	\abs{R_0}\leq\abs{\cS_{bad}}\leq\exp(-\xi n))\cdot\abs{\cS(\Phi_{t,\sigma})}.
	\end{equation}
Combining (\ref{eqIcySumm3}) and~(\ref{eqIcySumm4}) we see that the decomposition
$R_0,\ldots,R_N$ satisfies {\bf SH1}.
Furthermore, {\bf SH2} is satisfied by construction.
\qed\end{proof}

\subsection{Condensation}\label{Sec_cond}

Here we prove the second part of \Thm~\ref{Thm_shape}.
The following proposition reduces that task to a problem in calculus.

\begin{proposition}\label{Prop_cond}
Let $k\geq3$ and $r>0$ be fixed.
Let  $0<\theta\leq1$ and let
	$$\psi:(0,1)\rightarrow\RR,\quad\alpha\mapsto-\alpha\theta\ln\alpha-(1-\alpha)\theta\ln(1-\alpha)+
			r\ln\bc{1-\frac{1-(1-\alpha\theta)^k}{2^k-1}}.$$
If there is a number $a\in(0,1)$  such that
	\begin{equation}\label{eqpsi}
	\sup_{a<\alpha\leq1}\psi(\alpha)+\rho/2^k<0
	\end{equation}
then $\PHI_t$ is $2a\theta$-condensed. 
\end{proposition}
\begin{proof}
Let $(\Phi,\sigma)$ be a pair chosen from the planted distribution $\cP_k(n,m)$.
For $\alpha>0$ we let
	$$X_\alpha=\abs{\cbc{\tau\in\cS(\Phi_{t,\sigma}):\dist(\sigma,\tau)=\alpha\theta n}}.$$
Then
	$\Erw X_\alpha\leq\bink{\theta n}{\alpha\theta  n}\brk{1-\frac{1-(1-\alpha\theta)^k}{2^k-1}}^m$
and  taking logarithms we obtain
	$\frac1n{\ln\Erw X_\alpha}\leq\psi(\alpha).$
Hence, $\frac1n\ln\Erw X_\alpha<-\rho/2^k$ for $\alpha>a$ by~(\ref{eqpsi}).
Thus, by Markov's inequality we have
	\begin{eqnarray*}
	\pr\brk{\exists \tau\in\cS_t(\Phi_{t,\sigma}):d(\sigma,\tau)\geq a\theta n}
		\leq \theta n\cdot\exp(-(\Omega(1)+\rho/2^k)n)<\exp(-\rho n/2^k).
	\end{eqnarray*}
Therefore, the assertion follows from \Cor~\ref{Cor_SATExchange}.
\qed\end{proof}

\begin{lemma}\label{Lemma_cond}
Suppose that $\rho\leq k\ln2-2\ln k$ and $\theta=(1-1/\rho^2)\frac{\rho}{k\ln2}$.
Moreover, assume that $\rho\geq \rho_0$ and $k\geq k_0$ for certain constants $\rho_0,k_0$.
Let $a=\exp(2-\rho)$.
Then~(\ref{eqpsi}) is satisfied.
\end{lemma}
\begin{proof}
Let $h\bc\cdot$ be the entropy function.
We have
	\begin{eqnarray*}
	\psi(\alpha)&\leq&\theta h(\alpha)-\frac{\rho}k(1-\exp(-\alpha k\theta)).
	\end{eqnarray*}
To bound the r.h.s., we are going to consider several cases.
\begin{description}
\item[Case 1: $\alpha \leq1/(k\rho\theta)$.]
	As $\alpha\geq a=\exp(2-\rho)$, we obtain
	\begin{eqnarray*}
	\psi(\alpha)&\leq&\alpha\theta\brk{1-\ln\alpha-\rho+\alpha k\rho\theta/2}
		\leq\alpha\theta\brk{\frac{\alpha k\rho\theta}{2}-1}\leq-\alpha\theta/2.
	\end{eqnarray*}
	The assumption $\rho\leq k\ln2-2\ln k$ ensures that the last term is smaller than
		$-\rho/2^k$.
\item[Case 2: $1/(k\rho\theta)<\alpha<1/(k\theta)$.]
	We have
	\begin{eqnarray*}
	\psi(\alpha)&\leq&
		\alpha\theta\brk{1-\ln\alpha-\rho+\alpha k\rho\theta/2}\\
		&\leq&\alpha\theta\brk{1+\ln(k\rho\theta)-\rho+\frac{\alpha k\rho\theta}{2}}\\
		&\leq&\alpha\theta\brk{1+\ln(k\rho\theta)-\rho/2}\qquad\qquad\quad\mbox{[as $\alpha<1/(k\theta)$]}\\
		&\leq&\alpha\theta\brk{1-\ln\ln2+2\ln\rho-\rho/2}\qquad\mbox{ [as $\theta\leq\frac{\rho}{k\ln2}$]}\\
		&\leq&-\alpha\theta\rho/4.
	\end{eqnarray*}
	The assumption $\rho\leq k\ln2-2\ln k$ ensures that the last term is smaller than
		$-\rho/2^k$.
\item[Case 3: $1/(k\theta)<\alpha\leq\alpha_0=0.15$.]
	We have
		\begin{eqnarray*}
		\psi(\alpha)&\leq&\theta h(\alpha)-\frac{\rho}k(1-\exp(-\alpha k\theta))
			\leq\theta h(\alpha_0)-\frac\rho k(1-1/\eul)\\
			&\leq&\frac\rho k\brk{\frac{h(\alpha_0)}{\ln2}-1+1/\eul}.
		\end{eqnarray*}
	The choice of $\alpha_0$ ensures that the last term is smaller than $-\rho/2^k$.
\item[Case 4: $\alpha_0<\alpha$.]
	As $k\theta=(1-1/\rho^2)\rho/\ln2$, we get
		\begin{eqnarray*}
		\psi(\alpha)&\leq&\theta h(\alpha)-\frac{\rho}k(1-\exp(-\alpha k\theta))
			\leq\theta\ln2-\frac\rho k(1-\exp(-\alpha_0(1-1/\rho^2)\rho/\ln2))\\
			&\leq&\frac\rho k\brk{\exp(-\alpha_0\rho)-1/\rho^{2}}.
		\end{eqnarray*}
	The last term is smaller than $-\rho/2^k$.
\end{description}
\qed\end{proof}

\begin{proof}[\Thm~\ref{Thm_shape}, part~2]
Let $\theta_0=(1-1/\rho^2)\rho/(k\ln2)$ and $t_0=(1-\theta_0)n$.
Suppose that $\theta\geq\theta_0$.
Then $\PHI_t$ is obtained from $\PHI_{t_0}$ by assigning some further variables.
Therefore, 
	$$\max\cbc{d(\sigma,\tau):\sigma,\tau\in\cS(\PHI_t)}\leq
			\max\cbc{d(\sigma,\tau):\sigma,\tau\in\cS(\PHI_{t_0})}.$$
Hence, \Prop~\ref{Prop_cond} and \Lem~\ref{Lemma_cond} imply that $\PHI_t$ is $\exp(2-\rho)$-condensed \whp
\qed\end{proof}

\subsection{Pairwise distances}

Recall that $\Phi_{t,\sigma}$ denotes the formula obtained by substituting the values $\sigma(x_i)$ for $x_i$ for $1\leq i\leq t$.

\begin{lemma}\label{Lemma_overlap}
Suppose that $\theta\geq\frac\rho{k\ln 2}(1+1/\rho^2+k/2^{k-2})$.
Let $(\Phi,\sigma)$ be a pair chosen from the distribution $\cU_k(n,m)$. 
\Whp\ we have
	$$\abs{\cbc{\tau\in\cS(\Phi_{t,\sigma}):\dist(\tau,\sigma_t)\leq0.49\theta n}}\leq\exp(-\Omega(n))\abs{\cS_t(\Phi)}.$$
\end{lemma}
\begin{proof}
We need to work with the function
	$$\psi(\alpha)=-\alpha\theta\ln\alpha-(1-\alpha)\theta\ln(1-\alpha)+
			\frac{2^k\rho}{k}\ln\bc{1-\frac{1-(1-\alpha\theta)^k}{2^k-1}}.$$
By \Cor~\ref{Cor_count}, \whp\ $\frac1n\ln\abs{\cS(\Phi_t)}\geq\theta\ln2+r\ln\bc{1-2^{-k}}-\rho/2^k$.
From now on, we are going to work with the planted model $\cP_k'(n,m)$.
We are going to show that
	$$\sup_{\alpha\leq0.1}\psi(\alpha)-\theta\ln2-r\ln\bc{1-2^{-k}}<-\rho/2^{k-1}.$$
Then the assertion follows from \Cor~\ref{Cor_Pnm'}.
We have
	\begin{eqnarray*}
	\psi(\alpha)-\theta\ln2-\frac{2^k\rho}{k}\ln\bc{1-2^{-k}}&=&
			\theta(h(\alpha)-\ln2)+\frac{2^k\rho}k\ln\brk{1+\frac{(1-\alpha\theta)^k-2^{1-k}(1-(1-\alpha\theta)^k)}{2^k-1}}\\
		&\leq&\theta(h(\alpha)-\ln2)+\frac{\rho}{k}(1-\alpha\theta)^k+2^{-k}\\
		&\leq&\theta(h(\alpha)-\ln2)+\frac{\rho}{k}\exp(-\alpha k\theta)+2^{-k}.
	\end{eqnarray*}
The differential of the last expression with respect to $\theta$ is negative,
and thus the function is monotonically decreasing in $\theta$.
Therefore, it suffices to consider the minimum value $\theta=\rho/(k\ln2)$.
Thus, we obtain
	\begin{eqnarray*}
	\psi(\alpha)-\theta\ln2-\frac{2^k\rho}k\ln\bc{1-2^{-k}}&\leq&\frac{\rho}k\bc{\frac{h(\alpha)}{\ln2}-1+\exp(-\alpha \rho/\ln2)}+2^{-k}.
	\end{eqnarray*}
We consider a few different cases.
\begin{description}
\item[Case 0: $\alpha<\exp(2-\rho)$.]
	\Lem~\ref{Lemma_alpha} shows that
		$\psi(\alpha)\leq1/(k\rho)$ and (\ref{eqClusterEntropy}) shows that
			$$\theta\ln2+2^k\frac\rho k\ln(1-2^{-k})\geq\theta\ln2-\rho/k-\rho/2^k.$$
	Hence,
		\begin{eqnarray*}
		\psi(\alpha)-\theta\ln2-\frac{2^k\rho}k\ln\bc{1-2^{-k}}&\leq&\frac1{k\rho}-\theta\ln2+\frac\rho k+\rho/2^k.
		\end{eqnarray*}
	Since we are assuming that $\theta\geq\frac\rho{k\ln 2}(1+1/\rho^2+k/2^{k-2})$,
		the r.h.s.\ is smaller than $\rho/2^{k-1}$.
\item[Case 1: $\exp(2-\rho)\leq\alpha\leq\exp(-\rho/2)$.]
	Bounding the exponential by a quadratic function, we get
	\begin{eqnarray*}
	\psi(\alpha)-\theta\ln2-r\ln\bc{1-2^{-k}}&\leq&
		\frac{\alpha\rho}{ k\ln2}\brk{1-\ln\alpha-\rho+\frac{\alpha\rho^2}{4\ln2}}+2^{-k}\\
		&\leq&\frac{\alpha\rho}{ k\ln2}\brk{-1+\frac{(\alpha\rho)^2}{2\ln2}}+2^{-k}<-\rho/2^{k-1},\\
	\end{eqnarray*}
	provided that $\rho_0\leq \rho\leq k\ln2-2\ln k$.
\item[Case 2: $\exp(-\rho/2)\leq\alpha\leq1/(2\rho)$.]
	Bounding the exponential by a quadratic function, we get
	\begin{eqnarray*}
	\psi(\alpha)-\theta\ln2-r\ln\bc{1-2^{-k}}&\leq&
		\frac{\alpha\rho}{ k\ln2}\brk{1-\ln\alpha-\rho+\frac{(\alpha\rho)^2}{2\ln2}}+2^{-k}<-\rho/2^{k-1},
	\end{eqnarray*}
	provided that $\rho_0\leq \rho\leq k\ln2-2\ln k$.
\item[Case 3: $1/(2\rho)<\alpha\leq10\ln(2)/\rho$.]
	Suppose that $\alpha=x\ln(2)/\rho$ for some $1/2\leq x\leq10\ln2$.
	Then
	\begin{eqnarray*}
	\psi(\alpha)-\theta\ln2-r\ln\bc{1-2^{-k}}&\leq&
		\frac{\rho}{k}\brk{\frac{x\ln2}\rho(1-\ln x-\ln\ln2+\ln\rho)-1+\exp(-x)}+2^{-k}.
	\end{eqnarray*}
	As $x$ remains bounded away from $0$, the term $\exp(-x)-1$ is strictly negative.
	Thus, the entire expression is smaller than $-\rho/2^{k-1}$ for $\rho\geq\rho_0$ sufficiently large.
\item[Case 4: $10\ln(2)/\rho<\alpha\leq0.49$.]
	We have
	\begin{eqnarray*}
	\psi(\alpha)-\theta\ln2-r\ln\bc{1-2^{-k}}&\leq&\frac{\rho}k\bc{\frac{h(0.1)}{\ln2}-1+\exp(-10)}+2^{-k}.
	\end{eqnarray*}
	The r.h.s.\ is clearly smaller than $-\rho/2^{k-1}$.
\end{description}
\qed\end{proof}
\Lem~\ref{Lemma_overlap} directly implies the third part of \Thm~\ref{Thm_shape}.

\section{Proof of \Thm~\ref{Thm_marginals}}\label{Sec_marg}

\subsection{Bounding the marginals away from $0,1$}

Here we prove the first part of \Thm~\ref{Thm_marginals}.
We may assume that $\theta\geq\rho/(k\ln2)$.
The goal is to show that the marginals of a substantial fraction of the variables
$x_{t+1},\ldots,x_n$ are bounded away from $0,1$.

We set up an auxiliary graph $\cG$ whose vertices are all pairs $(x,\tau)$ of variables $x\in V_t$ and assignments $\tau\in\cS(\Phi_t)$.
A pair $(x,\tau)$ is connected by an edge with another pair $(x,\tau')$ if $\tau(x)=\tau'(x)$.
(Thus, the graph consists of components $(x,\cdot)$ with $x\in V_t$.)
\Lem~\ref{Lemma_overlap} implies the following.

\begin{corollary}
Let $\Phi_t$ is the formula obtained through the experiment {\bf U1}--{\bf U4}.
\Whp\ we have $2|E(\cG)|\leq0.511|\cS_t(\Phi)|^2\theta n$.
\end{corollary}
\begin{proof}
We count the number of edges from each assignment $\tau$.
By \Lem~\ref{Lemma_overlap}, almost all assignments $\tau$ are such that the `overlap'
with almost all other assignments $\tau'$ is at most $0.51\theta n$.
For such assignments, the number of edges incident with $\cbc{(\tau,x):x\in V_t}$ is bounded by $(1+o(1))0.51\theta n$.
\qed\end{proof}

To bound the marginals away from $0,1$, assume that indeed $2|E(\cG)|\leq0.511|\cS_t(\Phi)|^2\theta n$.
Any variable $x$ whose marginal does not lie in $(0.01,0.99)$ is such that the set
	$\cbc{(\tau,x):\tau\in\cS(\Phi_t)}$ induces at least $(1+o(1))0.99\cS(\Phi_t)^2/2$ edges.
Hence, if we let $\nu$ be the number of such variables, then
	$(1+o(1))0.99\cS(\Phi_t)^2\nu\leq2|E(\cG)|\leq0.511|\cS_t(\Phi)|^2\theta n$.
Hence, $\nu\leq\frac{0.511+o(1)}{0.99}\theta n\leq\frac23\theta n$.

\subsection{Concentration of the marginals about $0,1$}\label{Sec_Marg01}

To prove the second part of \Thm~\ref{Thm_marginals}, we need the following lemma.

\begin{lemma}\label{Lemma_rigidity}
Suppose that $\theta\leq\rho/(k\ln2)$.
Let $(\Phi,\sigma)$ be a pair chosen from the distribution $\cU_k\bc{n,m}$.
\Whp\ there is no set of variables $Z\subset V_t$ of size
	$2kn/2^k\leq |Z|\leq (\eul\rho)^{-4}\theta n$ such that
each variable in $Z$ supports two clauses under $\sigma$, each of which contains
an occurrence of a variable in $Z$ that evaluates to `false' under $\sigma$.
\end{lemma}
\begin{proof}
We work with the planted model $\cP_k'\bc{n,m}$.
Let $p$ be such that the expected number of clauses is $m$,
i.e., $(2^k-1)\bink{n}kp=m$.
Then the probability that a given set $Z$ of size $z$ is `bad' is bounded by
	$$\bc{z\bink{n}{k-2}p}^{2z}\leq(\alpha k\theta\rho)^{2z},\mbox{ with }\alpha=z/(\theta n).$$
Thus, the probability that there is a bad set of size $z$ is bounded by
	\begin{eqnarray*}
	\bink{\theta n}z(\alpha k\theta\rho)^{2z}&\leq&
		\bcfr{\eul\theta n}{z}^z\bc{\alpha k\theta\rho}^{2z}
			=\brk{\eul\alpha(k\theta\rho)^2}^z
			\leq(\eul\alpha\rho^4/\ln^22)^z\leq\exp(-z).
	\end{eqnarray*}
The assumption on $z$ ensures that this is sufficiently small
to move from the planted model to $\cU_k\bc{n,m}$
via \Cor~\ref{Cor_Pnm'}.
\qed\end{proof}

\begin{proof}[\Thm~\ref{Thm_marginals}, part~2]
If $k\theta<\ln(\rho)/2$, then the existence of forced variables immediately implies part~2 of \Thm~\ref{Thm_marginals}.
Thus, let us assume that $\ln(\rho)/2\leq k\theta\leq\rho/\ln2$.
Let $(\Phi,\sigma)$ be a pair chosen from the distribution $\cU_k(n,m)$.
Let $S$ be the set of rigid variables;
	by \Thm~\ref{Thm_vars}, we have $|S|\geq\rho^3\exp(-\rho)\theta n$ \whp\
Define an auxiliary bipartite graph as follows.
The vertices of the graph are the variables in $S$ and the satisfying assignments $\cS(\Phi_{t,\sigma})$.
Each variable $x\in S$ is connected with all $\tau\in\cS(\Phi_{t,\sigma})$ such that $\tau(x)\neq\sigma(x)$.
By the \Lem~\ref{Lemma_rigidity} and because
	$\Phi$ is $\exp(2-\rho)$-condensed (part~2 of \Thm~\ref{Thm_shape}), there is $\eps_k\ra0$ such that the number of edges of this bipartite graph is bounded
by $k2^{1-k}n\abs{\cS_t(\Phi)}$.
Hence, the degree sum over the variables satisfies
	$$\sum_{x\in S}d(x)\leq k2^{2-k}n\abs{\cS_t(\Phi)}.$$
We may assume without loss of generality that $\sigma=\vecone$ is the all-true assignment.
Then the marginal $\mu(x)$ equals $1-d(x)/\abs{\cS_t(\Phi)}$.
Hence, the above bound on the degree sum shows that $\mu(x)\leq 2^{-k/2}$ for all but $0.01\theta n$ variables $x\in S$.
\qed\end{proof}

\begin{proof}[\Thm~\ref{XThm_cond}, part~4]
This follows directly by applying \Lem~\ref{Lemma_rigidity} to the self-contained set 
obtained in Appendix~\ref{Sec_rigid}.
\qed\end{proof}

\section{Belief propagation}\label{Sec_BPproofs}

The proof of \Thm~\ref{XThm_NoBP} is based on results from~\cite{BPdec}.
These results show that, in order to obtain \Thm~\ref{XThm_NoBP},
we essentially have to verify that the outcome $\PHI_t$ of the experiment {\bf U1}--{\bf U4} enjoys certain quasi-randomness properties.
We begin by stating the necessary properties.
To this end, we define
	\begin{equation}\label{eqdeltat}
	\qquad\delta_t=\exp(-c\theta k),
	\end{equation}
where $c>0$ is a small absolute constant (independent of $k,r,t,n$).

Fix a $k$-CNF $\Phi$ and an assignment $\sigma\in\cbc{0,1}^V$.
Let $\Phi_{t,\sigma}$ denote the CNF obtained from $\Phi$ by substituting $\sigma(x_1),\ldots,\sigma(x_k)$ for $x_1,\ldots,x_t$ and simplifying.
Let $G=G(\Phi_{t,\id,\vecone})$ denote the factor graph.
For a variable $x\in V_{t}$ and a set $Q\subset V_{t}$ let
	\begin{eqnarray}
	N_{\leq1}(x,Q)&=&\{b\in N(x):\abs{N(x)\cap Q\setminus\cbc x}\leq1
			\wedge0.1\theta k\leq|N(b)|\leq10\theta k\}.
		\label{eqNleq1}
	\end{eqnarray}
Thus, $N_{\leq1}(x,Q)$ is the set of all clauses that contain $x$ (which may or may not be in $Q$) and at most
one other variable from $Q$.
In addition, there is a condition on the \emph{length} $|N(b)|$ of the clause $b$ in the decimated formula $\Phi_{t,\sigma}$.
Observe that having assigned the first $t$ variables, we should `expect' the average clause length to be $\theta k$.
For a linear map $\Lambda:\RR^{V_t}\ra\RR^{V_t}$ let $\cutnorm{\Lambda}$ signify the norm
	$$\cutnorm{\Lambda}=\max_{\zeta\in\RR^{V_t}\setminus\cbc0}\frac{\norm{\Lambda\zeta}_1}{\norm{\zeta}_\infty}.$$

\begin{definition}
Let $\delta>0$.
We say that $(\Phi,\sigma)$ is \emph{\bf\em ($\delta,t)$-quasirandom} if $\Phi$ satisfies {\bf Q0}
and $\Phi_{t,\sigma}$ satisfies {\bf Q1}--{\bf Q4} below.
\begin{description}
\item[\bf Q0.]  There are no more than $\ln\ln n$ redundant clauses.
			Moreover, no variable occurs in more than $\ln n$ clauses of $\Phi$.
\item[\bf Q1.]
	No more than $10^{-5}\delta\theta n$ variables occur in clauses
				of length less than $\theta k/10$ or greater than $10\theta k$.
	Moreover, there are at most $10^{-4}\delta\theta n$  variables $x\in V_{t}$
		such that
			$$\textstyle(\theta k)^3\delta\cdot \sum_{b\in N(x)}2^{-|N(b)|}>1.$$
\item[\bf Q2.]
	If $Q\subset V_{t}$ has size $\abs Q\leq\delta\theta n$, then there are no more than
		$10^{-4}\delta\theta n$ variables $x$ such that either
			\begin{eqnarray}
			\sum_{b\in N(x):|N(b)\cap Q\setminus\cbc x|=1}\hspace{-12mm}
					2^{-|N(b)|}&>&\rho(\theta  k)^5\delta,\mbox{ or}\label{Q2a}\\
			\sum_{b\in N(x):|N(b)\cap Q\setminus\cbc x|>1}\hspace{-12mm}
						2^{|N(b)\cap Q\setminus\cbc x|-|N(b)|}&>&\frac{\delta}{\theta  k},\mbox{ or }\label{Q2b}\\
			\abs{\sum_{b\in N_{\leq1}(x,Q)}\hspace{-2mm}\frac{\sign(x,b)}{2^{|N(b)|}}}&>&\frac{\delta}{1000}.\label{Q2c}
			\end{eqnarray}
\item[\bf Q3.] For any $0.01\leq z\leq 1$ and any set $Q\subset V_{t}$ of size
		$0.01\delta(n-t)\leq|Q|\leq 100\delta(n-t)$
			we have
				$$\sum_{b:|N(b)\cap Q|\geq z|N(b)|}|N(b)|\leq1.01|Q|/z.$$
\item[\bf Q4.] For any set $Q\subset V_{t}$ of size $|Q|\leq 10\delta(n-t)$ the linear operator
			\begin{eqnarray}\label{Q4}
			\Lambda_Q:\RR^{V_{t}}\rightarrow\RR^{V_{t}},&&	\Gamma\mapsto
				\bigg(\sum_{b\in N_{\leq1}(x,Q)}\sum_{y\in N(b)\setminus\cbc x}2^{-|N(b)|}\cdot
					\sign(x,b)\sign(y,b)\Gamma_y\bigg)_{x\in V_{t}}
			\end{eqnarray}
		has norm $\cutnorm{\Lambda_Q}\leq \delta^4\theta n$.
\end{description}
\end{definition}
With respect to {\bf Q0}, we have

\begin{lemma}[\cite{BPdec}]\label{Lemma_Q0}
The random formula $\PHI$ satisfies condition {\bf Q0} \whp,
for any density $0<r=m/n\leq2^k\ln2$.
\end{lemma}

Let $\Phi$ be a $k$-CNF and let $\delta>0$. 
For a number $\delta>0$ and an index $l> t$ we say that $x_{l}$ is {\em $(\delta,t)$-biased} if 
the result $\mu_{x_{l}}(\Phi_{t,\sigma},\omega)$ of the BP computation
on  $\Phi_{t,\sigma}$
differs from $\frac12$ by more than $\delta$, i.e.,
	$$\abs{\mu_{x_{l}}(\Phi_{t,\sigma},\omega)-1/2}>\delta.$$
Moreover, $(\Phi,\sigma)$ is {\em $(\delta,t)$-balanced} if no more than
$\delta \theta n$ variables are $(\delta,t)$-biased.

\begin{theorem}[\cite{BPdec}]\label{Thm_dynamics}
There is $\rho_0>0$ such that for any $k,r$ satisfying $\rho_0\cdot 2^k/k\leq r\leq2^k\ln2$ and $n$ sufficiently large
the following is true.
Suppose $(\Phi,\sigma)$ is $(\delta_t,t)$-quasirandom for some $1\leq t\leq T=(1-\ln(\rho)/(c^2k))n$.
Then $(\Phi,\sigma)$ is $(\delta_t,t)$-balanced.
\end{theorem}
At the end of this section, we will verify that random formulas chosen from the distribution $\cP_k'(n,m)$
are indeed quasirandom.

\begin{proposition}\label{Prop_reason}
There exists a constant $\rho_0>0$ such that for any $k,r$ satisfying $\rho_0\cdot2^k/k\leq r\leq2^k\ln2$
there is $\xi=\xi(k,r)>0$ so that for $n$ large and $\delta_t$, $T$ as in~\Thm~\ref{Thm_dynamics}
the following is true.
Let  $(\Phi,\sigma)$ be a pair chosen from the planted model $\cP_k'(n,m)$,
given that $\sigma=\vecone$ is the all-true assignment.
Then
	\begin{eqnarray*}
	\pr\brk{(\Phi,\sigma)\mbox{ is $(\delta_t,t)$-quasirandom}|\mbox{\bf Q0}}
		&\geq&
			1-\exp\brk{-\rho 2^{1-k}n}
	\end{eqnarray*}
 for any $1\leq t\leq T$.
\end{proposition}
Finally, \Thm~\ref{XThm_NoBP} follows by combining \Cor~\ref{XCor_SATExchange}, \Thm~\ref{Thm_dynamics}, and \Prop~\ref{Prop_reason}.

\subsubsection{Proof of \Prop~\ref{Prop_reason}.}

Let $\PHI'=\PHI'_k(n,m)$ be a random formula obtained by including each possible clause
with probability $p=m/(2^k\bink nk)$ independently.

\begin{proposition}[{\cite[Appendix~E]{BPdec}}]\label{Prop_reasonable}
There exists a constant $\rho_0>0$ such that for any $k,r$ satisfying $\rho_0\cdot2^k/k\leq r\leq2^k\ln2$
for $n$ large and $\delta_t$, $T$ as in~(\ref{eqdeltat}) the following properties
hold for a random formula $\PHI'$ with probability at least
	$1-\exp\brk{-10\sum_{s\leq t}\delta_s}$
for any $1\leq t\leq T$, given that $\PHI'$ satisfies {\bf Q0}.
\begin{enumerate}
\item {\bf Q1} and {\bf Q3} are satisfied.
\item For any set $Q$ of size $|Q|\leq\delta\theta n$ there are at most
	$10^{-5}\delta\theta n$ variables $x$ that satisfy either (\ref{Q2a}), (\ref{Q2b}), or
		\begin{equation}\label{Q2c'}
		\abs{\sum_{b\in N_{\leq1}(x,Q)}\hspace{-2mm}\frac{\sign(x,b)}{2^{|N(b)|}}}>\frac{\delta}{2000}.
		\end{equation}
\item For any $Q$ the operator $\Lambda_Q$ from~(\ref{Q4}) satisfies
		$\cutnorm{\Lambda_Q}\leq \delta^4(n-t)/2$
\end{enumerate}
\end{proposition}

Let $\PHI_t$ be the formula obtain from $\PHI$ by substituting the value `true' for $x_1,\ldots,x_{t-1}$ and simplifying.
Since the $\delta_s$ form a geometric sequence, we have
	\begin{eqnarray*}
	\Sigma_t&=&\sum_{s\leq t}\delta_s\sim \frac n{ck\exp(c\theta k)}.
	\end{eqnarray*}
Observe that
	\begin{eqnarray*}
	\theta\delta n&>&10^{15}\Sigma_t
	\end{eqnarray*}
if $\rho\geq\rho_0$ is chosen sufficiently large.

\begin{lemma}\label{Lemma_add}
There exists a constant $\rho_0>0$ such that for any $k,r$ satisfying $\rho_0\cdot2^k/k\leq r\leq2^k\ln2$
the following is true for the random formula $\PHI'$ with probability at least
	$1-\exp(-\rho2^{2-k}n).$
\begin{enumerate}
\item The total number of all-negative clauses is bounded by $2^{1-k}m$.
\item For each variable $x\in V_t$ let $N_x$ be the number of all-negative clauses in which $x$ appears.
		Then the number of variables $x\in V_t$ with $N_x>2^{0.01\theta k}$ is bounded by
		$\delta^2\theta n$.
\end{enumerate}
\end{lemma}
\begin{proof}
The first assertion simply follows from Chernoff bounds.
With respect to the second assertion, assume that the first claim occurs, i.e., the total number of all-negative
clauses is bounded by $2^{1-k}m=2\rho n/k$.
Then for each variable the average number of occurrences in such clauses is bounded by $2\rho$.
Therefore, the total number of variables that occur more than $2^{0.01\theta k}$ times is bounded by $2\rho\cdot2^{-0.01\theta k}n$.
By symmetry, the number of such variables that are amongst the last $\theta n$ variables is (asymptotically)
binomially distribution with mean $2\rho\cdot2^{-0.01\theta k}\theta n$.
Therefore, the second assertion follows from Chernoff bounds.
\qed\end{proof}

\begin{proof}[\Prop~\ref{Prop_reason}]
Let $(\Phi,\sigma)$ be a random pair chosen from the distribution $\cP'_k(n,m)$.
We may assume without loss of generality that $\sigma$ is the all-true assignment.
Thus, the formula $\Phi$ is obtained by including each clause that does not consist of negative literals
only with probability $p=m/((2^k-1)\bink nk)$ independently.
Now, let $\Phi'$ be the formula obtained by addition to $\Phi$ each of the $\bink nk$ all-negative clauses independently with probability $p$.
Then $\Phi'$ has distribution $\PHI'_k(n,m\cdot\frac{2^k}{2^k-1})$.
Thus, with probability at least 
	$1-\exp[-10\sum_{s\leq t}\delta_s]$
the formula $\Phi'$ has the properties 1.--3.\ from \Prop~\ref{Prop_reasonable}.
Let us condition on this event.

Since $\Phi'$ contains $\Phi$ as a sub-formula, the fact that $\Phi'$ enjoys properties {\bf Q1} and {\bf Q3}
implies that the same is true of $\Phi$.
Furthermore, any variable $x$ for which either~(\ref{Q2a}) or (\ref{Q2b}) is true in $\Phi$
has the same property in $\Phi'$ (because the expressions on the left hand side are monotone with respect to the addition of clauses).
With respect to the expression in~(\ref{Q2c}), we decompose the sum for the pair $(\Phi,\sigma)$ as
	$$S_x(\Phi,\sigma)=S_x(\Phi',\sigma)-R_x,$$
where $R_x$ sums over all clauses that are in $\Phi'$ but not in $\Phi$.
Due to {\bf Q1}, we may assume that only clauses of length at least $0.1\theta k$ occur in the sum $R_x$.
Thus, letting $N_x$ denote the number of clauses in $\Phi'\setminus\Phi$ containing $x$,
we get $|R_x|\leq2^{-0.1\theta k}N_x$.
The second part of \Lem~\ref{Lemma_add} implies that for all but $\delta^2\theta n$ variables
we have $N_x\leq2^{0.01\theta k}$.
Hence, $R_x$ is tiny for all but $\delta^2\theta n$ variables.
This shows that $\Phi$ satisfies {\bf Q2}.

With respect to {\bf Q4}, let $D$ be the difference of the two linear operators for $\Phi$ and $\Phi'$.
Only clauses of length at least $0.1\theta k$ and at most $10\theta k$ contribute to $D$.
Hence, letting $N$ denote the number of all-negative clauses, we have
	\begin{eqnarray*}
	\cutnorm D&\leq&2^{-0.1\theta k}(10\theta k)^2N.
	\end{eqnarray*}
Since $N\leq2^{1-k}m=2\rho n/k$ by \Lem~\ref{Lemma_add},
we thus get
	$$\cutnorm D\leq200\theta n(\theta k)2^{-0.1\theta k}.$$
Hence, the third part of \Prop~\ref{Prop_reasonable} implies that $\Phi$ satisfies {\bf Q4}.
\qed\end{proof}

\end{appendix}

\end{document}